\documentclass[a4paper,english,sumlimits,reqno]{amsart}

\usepackage[latin1]{inputenc}
\usepackage[T1]{fontenc}
\usepackage{babel}

\usepackage{xifthen}
\usepackage{ifthenx}

\usepackage{xargs}

\usepackage{xcolor}
\usepackage{graphicx}

\usepackage[colorinlistoftodos,prependcaption,textsize=tiny]{todonotes}
\newcommandx{\niels}[2][1=]{\todo[linecolor=red,backgroundcolor=red!25,bordercolor=red,#1]{#2}}
\newcommandx{\richard}[2][1=]{\todo[linecolor=blue,backgroundcolor=blue!25,bordercolor=blue,#1]{#2}}
\newcommandx{\paul}[2][1=]{\todo[linecolor=orange,backgroundcolor=orange!25,bordercolor=orange,#1]{#2}}

\newcommandx{\nielsin}[2][1=]{\niels[inline, caption={Niels}, #1]{%
\begin{minipage}{\textwidth-20pt}#2\end{minipage}}}
\newcommandx{\richardin}[2][1=]{\richard[inline, caption={Richard}, #1]{%
\begin{minipage}{\textwidth-30pt}#2\end{minipage}}}
\newcommandx{\paulin}[2][1=]{\paul[inline, caption={Paul}, #1]{%
\begin{minipage}{\textwidth-30pt}#2\end{minipage}}}

\usepackage{pstricks}
\usepackage{pstricks-add}
\usepackage{pst-node}
\usepackage[inline]{asymptote}

\usepackage[all]{xy}
\newcommand{\vcxymatrix}[1]{\vcenter{\xymatrix{#1}}}

\usepackage{fp}

\usepackage[section]{placeins}
\usepackage{float}

\usepackage{enumerate}

\usepackage{hyperref}
\usepackage{url}

\usepackage{amssymb}
\usepackage{amsmath}
\usepackage{amsthm}
\usepackage{mathtools}
\usepackage{exscale}
\usepackage{relsize}
\usepackage{bbm}
\usepackage{bm}

\usepackage[mathcal]{eucal}
\usepackage{mathrsfs}

\usepackage{stmaryrd}

\newcounter{proof}
\newenvironment{myproof}%
{\stepcounter{proof}\begin{proof}}%
{\end{proof}}%
\newcounter{proofstep}[proof]
\newenvironment{proofstep}[1][]%
{\refstepcounter{proofstep}\smallskip\par\noindent%
  \ifthenelse{\isempty{#1}}
    {\textsc{Step \theproofstep.}}
    {\textsc{#1.}}
  \noindent}%
{\par}%
\newcounter{proofcase}[proof]
\newenvironment{proofcase}[1][]%
{\refstepcounter{proofcase}\smallskip\par\noindent%
  \ifthenelse{\isempty{#1}}
    {\textsc{Case \theproofcase.}}
    {\textsc{#1.}}
  \noindent}%
{\par}%

\theoremstyle{plain}
\newtheorem{thm}{Theorem}[section]
\newtheorem*{thm*}{Theorem}

\newtheorem{pro}[thm]{Proposition}
\newtheorem{cor}[thm]{Corollary}
\newtheorem{lem}[thm]{Lemma}
\newtheorem{que}[thm]{Question}
\theoremstyle{definition}
\newtheorem{dfn}[thm]{Definition}
\newtheorem{rem}[thm]{Remark}
\theoremstyle{remark}

\numberwithin{equation}{section}

\DeclareMathOperator{\union}{\cup}

\DeclareMathOperator{\isect}{\cap}

\newcommand{\umd}{\ensuremath{\text{UMD}}}
\newcommand{\bmo}[1][]{\ensuremath{\text{BMO}_{#1}}}

\newcommand{\charfun}{\scalebox{1.0}{\ensuremath{\mathbbm 1}}}

\newcommand{\dint}{\mathscr{D}}
\newcommand{\drec}{\mathscr{R}}

\newcommand{\dif}{\ensuremath{\, \mathrm d}}

\DeclareMathOperator{\spn}{span}

\DeclareMathOperator{\cond}{\mathbb{E}}

\DeclareMathOperator{\lesslex}{<_\ell}

\DeclareMathOperator{\drless}{\lhd\,}
\DeclareMathOperator{\drlesseq}{\unlhd\,}
\newcommand{\drindex}{\mathcal{O}_{\drless}}

\begin{document}

\title[Factorization of the identity operator]{Factorization of the identity through operators with
  large diagonal}

\author[N.J.~Laustsen]{Niels Jakob Laustsen}
\address{Niels Jakob Laustsen,
  Department of Mathematics and Statistics, Fylde College, Lancaster
  University, Lancaster LA1 4YF, United Kingdom}
\email{n.laustsen@lancaster.ac.uk}

\author[R.~Lechner]{Richard Lechner}
\address{Richard Lechner,
  Institute of Analysis,
  Johannes Kepler University Linz,
  Altenberger Strasse 69,
  A-4040 Linz, Austria}
\email{Richard.Lechner@jku.at}

\author[P.F.X.~M\"uller]{Paul F.X.~M\"uller}
\address{Paul F.X.~M\"uller,
  Institute of Analysis,
  Johannes Kepler University Linz,
  Altenberger Strasse 69,
  A-4040 Linz, Austria}
\email{Paul.Mueller@jku.at}

\date{\today}



\thanks{The research of Lechner and M\"uller is supported by the Austrian Science Foundation (FWF)
  Pr.Nr.\ P23987 P22549}

\begin{abstract}  
  Given a Banach space~$X$ with an unconditional basis, we consider the following question: does the
  identity on~$X$ factor through every operator on~$X$ with large diagonal relative to the
  unconditional basis?  We show that on Gowers' unconditional Banach space, there exists an operator
  for which the answer to the question is negative.  By contrast, for any operator on the mixed-norm
  Hardy spaces $H^p(H^q)$, where $1 \leq p,q < \infty$, with the bi-parameter Haar system, this
  problem always has a positive solution.  The spaces  $L^p, 1 < p < \infty$,  were treated first by
  Andrew~[{\em Studia Math.}~1979].
\end{abstract}

\maketitle

\makeatletter
\providecommand\@dotsep{5}
\def\listtodoname{List of Todos}
\def\listoftodos{\@starttoc{tdo}\listtodoname}
\makeatother

\section{Introduction}\label{sec:intro}

\noindent
Let~$X$ be a Banach space. A \emph{basis} for~$X$ will always mean a Schauder basis.  We denote
by~$I_X$ the identity operator on~$X$, and write $\langle\cdot,\cdot\rangle$ for the bilinear
duality pairing between~$X$ and its dual space~$X^*$.  By an \emph{operator} on~$X$, we understand a
bounded and linear mapping from~$X$ into itself.

Suppose that $X$ has a normalized basis $(b_n)_{n\in\mathbb N}$, and let $b_n^*\in X^*$ be the
$n^{\text{th}}$ co\-or\-di\-nate functional. For an operator~$T$ on~$X$, we say that:
\begin{itemize}
\item $T$ \emph{has large diagonal} if $\inf_{n\in\mathbb N} |\langle Tb_n,b_n^*\rangle| >0$;
\item $T$ is \emph{diagonal} if $\langle Tb_m,b_n^*\rangle = 0$ whenever $m,n\in\mathbb N$ are
  distinct;
\item \emph{the identity operator on $X$ factors through~$T$} if there are operators~$R$ and~$S$
  on~$X$ such that the diagram
  \begin{equation*}
    \vcxymatrix{X \ar[r]^{I_X} \ar[d]_R & X\\
      X \ar[r]_T & X \ar[u]_S}
  \end{equation*}
  is commutative.
\end{itemize}
Suppose that the basis $(b_n)_{n\in\mathbb N}$ for~$X$ is unconditional.  Then the diagonal
operators on~$X$ correspond precisely to the elements of~$\ell_\infty(\mathbb N)$, and so for each
operator~$T$ on~$X$ with large diagonal, there is a diagonal operator~$S$ on~$X$ such that
$\langle STb_n,b_n^*\rangle = 1$ for each $n\in\mathbb N$.  This observation naturally leads to the
following question.
\begin{que}\label{mainquestion}
  Can the identity operator on~$X$ be factored through each operator on~$X$ with large diagonal?
\end{que}

In classical Banach spaces such as $\ell^p$ with the unit vector basis and $L^p$ with the Haar
basis, the answer to this question is known to be positive.  These are the theorems of
Pe{\l}czy{\'n}ski~\cite{pelczynski:1960} and Andrew~\cite{andrew:1979}, respectively; see also
Johnson, Maurey, Schechtman and Tzafriri~\cite[Chapter~9]{jmst:1979}.

The aim of the present paper is to establish the following two results.
\begin{itemize}
\item There exists a Banach space with an unconditional basis for which the answer to
  Question~\ref{mainquestion} is negative. This result relies heavily on the deep work of
  Gowers~\cite{gowers:1994} and Gowers-Maurey~\cite{gowers_maurey:1997}.
\item Question~\ref{mainquestion} has a positive answer for the mixed-norm Hardy space $H^p(H^q)$,
  where $1 \leq p,q <\infty$, with the bi-parameter Haar system as its unconditional basis. This
  conclusion can be viewed as a bi-parameter extension of Andrew's theorem~\cite{andrew:1979} on the
  perturbability of the one-parameter Haar system in~$L^p$.
\end{itemize}
The precise statements of these results, together with their proofs, are given in
Sections~\ref{sec:large} and~\ref{sec:andrew}, respectively.

\subsection*{Acknowledgements}
It is our pleasure to thank Th. Schlumprecht for very informative conversations and for encouraging
the collaboration between Lancaster and Linz.  Special thanks are due to J. B. Cooper (Linz) for
drawing our attention to the work of Andrew~\cite{andrew:1979}.

\section{The answer to Question~\ref{mainquestion} is not always positive}\label{sec:large}
\noindent
The aim of this section is to establish the following result, which answers
Ques\-tion~\ref{mainquestion} in the negative.
\begin{thm}\label{operatorlargediagonalnotfactorid}
  There is an operator~$T$ on a Banach space~$X$ with an unconditional basis such that $T$ has large
  diagonal, but the identity operator on~$X$ does not factor through~$T$.
\end{thm}

The proof of Theorem~\ref{operatorlargediagonalnotfactorid} relies on two ingredients.  The first of
these is Fredholm theory, which we shall now recall the relevant parts of.

Given an operator~$T$ on a Banach space~$X$, we set
\begin{equation*}
  \alpha(T) = \dim\ker T\in\mathbb N_0\cup\{\infty\}
  \qquad\text{and}\qquad
  \beta(T) = \dim(X/T(X))\in\mathbb N_0\cup\{\infty\},
\end{equation*}
and we say that:
\begin{itemize}
\item $T$ is an \emph{upper semi-Fredholm operator} if $\alpha(T)<\infty$ and $T$ has closed range;
\item $T$ is a \emph{Fredholm operator} if $\alpha(T)<\infty$ and $\beta(T)<\infty$.
\end{itemize}
Note that the condition $\beta(T)<\infty$ implies that~$T$ has closed range (see,
\emph{e.g.,}~\cite[Corollary~3.2.5]{caradus_pfaffenberger_yood:1974}), so that each Fredholm
operator is automatically upper semi-Fred\-holm. For an upper semi-Fredholm operator~$T$, we define
its \emph{index} by
\begin{equation*}
  i(T) = \alpha(T) - \beta(T)\in\mathbb Z\cup\{-\infty\}.
\end{equation*}

The main property of the class of upper semi-Fredholm operators that we shall require is that it is
stable under strictly singular perturbations in the following precise sense. Let $T$ be an upper
semi-Fredholm operator on a Banach space~$X$, and suppose that $S$ is an operator on~$X$ which is
\emph{strictly singular} in the sense that, for each $\varepsilon >0$, every infinite-dimensional
subspace of~$X$ contains a unit vector~$x$ such that $\|Sx\|\leqslant\varepsilon$. Then $T+S$ is an
upper semi-Fredholm operator, and
\begin{equation*}
  i(T+S) = i(T).
\end{equation*} A proof of this result can be found in~\cite[Proposition~2.c.10]{lindenstrauss-tzafriri:1977}.

We shall require the following piece of notation in the proof of our next lemma. For an element~$x$
of a Banach space~$X$ and a functional $f\in X^*$, we write $x\otimes f$ for the rank-one operator
on~$X$ defined by
\begin{equation*}
  (x\otimes f)y = \langle y,f\rangle x\qquad (y\in X).
\end{equation*}

\begin{lem}\label{diagsemifredholmindex0}
  Let $T$ be a diagonal upper semi-Fredholm operator on a Banach space with a basis. Then
  $\beta(T) = \alpha(T)$, so that $T$ is a Fredholm operator with index~$0$.
\end{lem}

\begin{proof} Let $X$ be the Banach space on which $T$ acts, and let $(b_n)_{n\in\mathbb N}$ be the
  basis for~$X$ with respect to which~$T$ is diagonal.  Set $N = \{n\in\mathbb N : Tb_n =
  0\}$. Since $T$ is diagonal, we have $\ker T = \overline{\operatorname{span}}\{ b_n : n\in N\}$,
  and so the set~$N$ is finite, with~$\alpha(T)$ elements. Consequently, we can define a projection
  of~$X$ onto $\ker T$ by \mbox{$P_N = \sum_{n\in N}b_n\otimes b_n^*$}.  The fact that
  $\ker P_N = \overline{\operatorname{span}}\{ b_n : n\in\mathbb N\setminus N\}$ implies that
  $T(X)\subseteq\ker P_N$.  Conversely, for each $n\in\mathbb N\setminus N$, we have
  $b_n = T(\langle Tb_n, b_n^*\rangle^{-1} b_n)$, so we conclude that $\ker P_N\subseteq T(X)$
  because $T$ has closed range.  Hence
  \begin{equation*}
    \beta(T) = \dim P_N(X) = \alpha(T)<\infty,
  \end{equation*} and the result follows.
\end{proof}

The other main ingredient in the proof of Theorem~\ref{operatorlargediagonalnotfactorid} is the
Banach space~$X_{\text{G}}$ which Gowers~\cite{gowers:1994} created to solve Banach's hyperplane
problem. This Banach space has sub\-sequently been investigated in more detail by Gowers and
Maurey~\cite[Sec\-tion~(5.1)]{gowers_maurey:1997}.  Its main properties are as follows.

\begin{thm}[Gowers~\cite{gowers:1994}; Gowers and
  Maurey~\cite{gowers_maurey:1997}]\label{hyperplanecounterex} There
  is a Banach space~$X_{\normalfont{\text{G}}}$ with an unconditional basis such that each operator
  on~$X_{\normalfont{\text{G}}}$ is the sum of a diagonal operator and a strictly singular operator.
\end{thm}

\begin{cor}\label{semiFredholmindex0}
  Each upper semi-Fredholm operator on the Banach space~$X_{\normalfont{\text{G}}}$ is a Fredholm
  operator of index~$0$.
\end{cor}

\begin{proof}
  Let $T$ be an upper semi-Fredholm operator on~$X_{\normalfont{\text{G}}}$. By
  Theorem~\ref{hyperplanecounterex}, we can find a diagonal operator~$D$ and a strictly singular
  operator~$S$ on~$X_{\normalfont{\text{G}}}$ such that $T = D+S$.  The stability of the class of
  upper semi-Fredholm operators under strictly singular perturbations that we stated above implies
  that~$D$ is an upper semi-Fredholm operator with the same index as~$T$, and hence the conclusion
  follows from Lemma~\ref{diagsemifredholmindex0}.
\end{proof}

\begin{proof}[Proof of Theorem~{\normalfont{\ref{operatorlargediagonalnotfactorid}}}.]  Let
  $X=X_{\normalfont{\text{G}}}$ be the Banach space from Theorem~\ref{hyperplanecounterex}, and let
  $(b_n)_{n\in\mathbb N}$ be the unconditional basis for~$X_{\normalfont{\text{G}}}$ with respect to
  which each operator on~$X_{\normalfont{\text{G}}}$ is the sum of a diagonal operator and a
  strictly singular operator.  We may suppose that $(b_n)_{n\in\mathbb N}$ is normalized.  Set
  \begin{equation*}
    T = I_{X_{\normalfont{\text{G}}}} + b_1\otimes b_2^* + b_2\otimes b_1^*.
  \end{equation*} Then $T$ has large
  diagonal because $\langle Tb_n,b_n^*\rangle =1$ for each $n\in\mathbb N$.

  Assume towards a contradiction that $I_{X_{\normalfont{\text{G}}}} = STR$ for some operators~$R$
  and~$S$ on~$X_{\normalfont{\text{G}}}$. Then~$R$ is injective, and its range is complemented
  (because $RST$ is a projection onto it), and it is thus closed, so that $R$ is an upper
  semi-Fredholm operator with $\alpha(R) = 0$. This implies that~$R$ is a Fredholm operator of
  index~$0$ by Corollary~\ref{semiFredholmindex0}, and hence $R$ is invertible.  Since $ST$ is a
  left inverse of~$R$, the uniqueness of the inverse shows that $R^{-1} = ST$, but this contradicts
  that the operator~$T$ is not injective (because \mbox{$T(b_1 - b_2) = 0$}).
\end{proof}

As we have seen in the proof of Theorem~\ref{operatorlargediagonalnotfactorid}, the identity
operator need not factor through a Fredholm operator. If, however, we allow ourselves sums of two
operators, then we can always factor the identity operator, as the following result shows.
\begin{pro}
  Let $T$ be a Fredholm operator on an infinite-dimensional Banach space $X$. Then there are
  operators $R_1$, $R_2$, $S_1$, and $S_2$ on~$X$ such that
  \begin{equation*}
    I_X = S_1TR_1 + S_2TR_2.
  \end{equation*}
\end{pro}
\begin{proof}
  Let $P = \sum_{j=1}^n x_j\otimes f_j$ be a projection of~$X$ onto the kernel of~$T$, where
  $n\in\mathbb N$, $x_1,\ldots, x_n\in X$, and $f_1,\ldots, f_n\in X^*$, and let $Q$ be a projection
  of~$X$ onto the range of~$T$.  Since this range is infinite-dimensional, we can find
  $y_1,\ldots,y_n\in X$ and $g_1,\ldots, g_n\in X^*$ such that
  $\langle Ty_j, g_k\rangle = \delta_{j,k}$ (the Kronecker delta) for each
  \mbox{$j,k\in\{1,\ldots,n\}$}. The restriction
  $\widetilde{T}\colon x\mapsto Tx,\, \ker P\to T(X)$, is invertible, so we may define an operator
  on~$X$ by $S_1 = J\widetilde{T}^{-1}Q$, where $J\colon \ker P\to X$ is the inclusion. Set
  \begin{equation*}
    R_1 = I_X - P,\qquad R_2 = \sum_{j=1}^n y_j\otimes f_j,\qquad\text{and}\qquad S_2 =
    \sum_{k=1}^n x_k\otimes g_k.
  \end{equation*} Then, for each $z\in X$, we have
  \begin{align*}
    (S_1TR_1 + S_2TR_2)z
    &= J\widetilde{T}^{-1}QT(z - Pz) +
      \sum_{j,k=1}^n \langle Ty_j,g_k\rangle \langle z,f_j\rangle x_k\\
    & = (z - Pz) + Pz = z,
  \end{align*}
  from which the conclusion follows.
\end{proof}

\section{The answer to Question~\ref{mainquestion} is positive in mixed-norm Hardy
  spaces}\label{sec:hardy}

\noindent
In many classical Banach spaces, the answer to Question~\ref{mainquestion} is known to be positive.
This includes $\ell^p$, $p\geq 1$, and $L^p$, $p>1$, see Pe{\l}czy{\'n}ski~\cite{pelczynski:1960}
and Andrew~\cite{andrew:1979}, respectively.  Closely related to this question is the work of
Johnson, Maurey, Schechtman and Tzafriri~\cite[Chapter~9]{jmst:1979}, in which they specify a
criterion for an operator on a rearrangement invariant function space to be a factor of the
identity.

We now turn to defining the mixed-norm Hardy spaces together with an unconditional basis, the
bi-parameter Haar system.  Let $\dint$ denote the collection of dyadic intervals given by
\begin{equation*}
  \dint
  = \{[k2^{-n},(k+1)2^{-n})\, :\, n,k\in \mathbb N_0, 0\leq k \leq 2^n-1\}.
\end{equation*}
The dyadic intervals are nested, i.e. if $I,J\in \dint$, then $I\cap J \in \{I,J,\emptyset\}$. For
$I\in \dint$ we let $|I|$ denote the length of the dyadic interval $I$. Let $I\in \dint$ and
$I\neq [0,1)$, then $\widetilde I$ is the unique dyadic interval satisfying $\widetilde I\supset I$
and $|\widetilde I| = 2 |I|$. Given $N_0\in \mathbb N_0$ we define
\begin{equation*}
  \dint_{N_0} = \{I\in \dint\, :\, |I| = 2^{-N_0}\}
  \qquad\text{and}\qquad
  \dint^{N_0} = \{I\in \dint\, :\, |I| \geq 2^{-N_0}\}.
\end{equation*}
Let~$h_I$ be the $L^\infty$-normalized Haar function supported on $I\in\dint$; that is, for
$I = [a,b)$ and $c = (a+b)/2$, we have $h_I(x) = 1$ if $a\le x <c$, $h_I(x) = -1$ if $c\leq x < b$,
and $h_I(x) = 0$ otherwise. Moreover, let $\drec = \{I\times J: I,J\in\dint\}$ be the collection of
dyadic rectangles contained in the unit square, and set
\begin{equation*}
  h_{I\times J}(x,y) = h_I(x)h_J(y),\qquad (I\times J\in\drec,\, x,y\in[0,1)).
\end{equation*}

For $1\leq p,q < \infty$, the \emph{mixed-norm Hardy space} $H^p(H^q)$ is the completion of
\begin{equation*}
  \spn\{ h_{I\times J}\, :\, I\times J \in \drec \}
\end{equation*}
under the square function norm
\begin{equation}\label{HpHqnorm}
  \|f\|_{H^p(H^q)}
  = \bigg(
  \int_0^1 \Big(
  \int_0^1 \big(
  \sum_{I\times J} |a_{I\times J}|^2 h_{I\times J}^2(x,y)
  \big)^{q/2}
  \dif y
  \Big)^{p/q}
  \dif x
  \bigg)^{1/p}
  ,
\end{equation}
where $f = \sum_{I\times J} a_{I\times J} h_{I\times J}$.  Then
$(h_{I\times J})_{I\times J\in\drec}$ is a $1$-unconditional basis of~$H^p(H^q)$, called the
\emph{bi-parameter Haar system}.  We begin with the following facts:
\begin{itemize}
\item It is recorded by Capon~\cite{capon:1982} that the identity operator provides an isomorphism
  between $H^p(H^q)$ and $L^p(L^q)$, $1 < p,q < \infty$.

\item Since the bi-parameter Haar system $\{h_{I\times J} : I\times J\in\drec\}$ is an unconditional
  basis, we do not need to specify an ordering of its index set~$\drec$.

\item This basis is $L^\infty$-normalized and not normalized in $H^p(H^q)$; we have
  \begin{equation*}
    \|h_{I\times J}\|_{H^p(H^q)} = |I|^{1/p}|J|^{1/q}.
  \end{equation*}
\end{itemize}
An operator $T\, :\, H^p(H^q)\to H^p(H^q)$ has large diagonal with respect to the
$L^\infty$-normalized Haar system $\{h_{I\times J} : I\times J\in\drec\}$ if and only if for some
$\delta > 0$ we have that $|\langle T h_{I\times J}, h_{I\times J} \rangle| \geq \delta |I\times J|$
for all $I\times J \in \drec$.  The remaining sections of the paper are devoted to proving the
following theorem.

\begin{thm}\label{thm:2d-andrew}
  Let $1 \leq p,q < \infty$ and $\delta > 0$, and let $T : H^p(H^q)\to H^p(H^q)$ be an operator
  satisfying
  \begin{equation*}
    |\langle T h_{I\times J}, h_{I\times J} \rangle| \geq \delta |I\times J|
    \qquad\text{for all $I\times J\in \drec$}.
  \end{equation*}
  Then the identity operator on $H^p(H^q)$ factors through~$T$, that is, there are opera\-tors~$R$
  and~$S$ such that the diagram
  \begin{equation}\label{eq:2d-andrew}
    \vcxymatrix{H^p(H^q) \ar[r]^{I_{H^p(H^q)}} \ar[d]_R & H^p(H^q)\\
      H^p(H^q) \ar[r]_T & H^p(H^q) \ar[u]_S}
  \end{equation}
  is commutative. Moreover, for any $\eta\in(0,1]$ the operators $R$ and $S$ can be chosen such that
  $\|R\|\|S\| \leq (1+\eta)/\delta$.
\end{thm}
For related, local (finite dimensional, quantitative) factorization theorems in bi-parameter $H^1$
and $\bmo$, see~\cite{mueller:2005, lechner_mueller:2014}.  Recently in~\cite{lechner:2016:1}, the
second named author obtained local factorization results in mixed-norm Hardy and $\bmo$ spaces by
combining methods of the present paper with techniques of~\cite{lechner_mueller:2014}. Despite the
fact that the constants in our theorem are independent of $p$ and $q$, we remark that the passage to
the non-separable limiting spaces (corresponding to $p=\infty$ or $q=\infty$) cannot be deduced
routinely from the proof given below. The non-separable space $SL^\infty$ consisting of functions
with square function in $L^\infty$ would be an example of such a limiting space. Factorization
theorems in $SL^\infty$ are treated by the second named author in~\cite{lechner:2016:2}.

The cornerstones upon which the constructions of the operators $R,S$ in Theorem~\ref{thm:2d-andrew}
rest are embeddings and projections onto a carefully chosen block basis of the bi-parameter Haar
system in mixed-norm Hardy spaces.

\section{Capon's local product condition and its consequences}\label{sec:projections}

\noindent
In this section, we treat embeddings and projections in $H^p(H^q)$.  They are the main pillars of
the construction underlying the proof of Theorem~\ref{thm:2d-andrew}.  We begin by listing some
elementary and well known facts concerning $H^p(H^q)$ and its dual.

\subsection{Basic facts and notation}\label{subsec:basic-facts}\hfill

\noindent
Let $1\leq p,q < \infty$ and let $H^p(H^q)^*$ denote the dual space of $H^p(H^q)$, identified as a
space of functions on $[0,1)^2$.  Then the duality pairing between $H^p(H^q)$ and $H^p(H^q)^*$ is
given by
\begin{equation*}
  \langle f, g\rangle
  = \int_0^1 \int_0^1 f(x,y)g(x,y)\dif y\dif x.
\end{equation*}
Correspondingly, we have
\begin{equation*}
  \|g\|_{H^p(H^q)^*} = \sup_{\|f\|_{H^p(H^q)}\leq 1} |\langle f, g\rangle|.
\end{equation*}
Since $h_{I\times J}$, $I\times J\in \drec$ is a $1$-unconditional Schauder basis in $H^p(H^q)$, we
may identify an element $g\in H^p(H^q)^*$ with the sequence
$(\langle h_{I\times J}, g\rangle)_{I\times J}$.  In the dual space, the norm of
$(|\langle h_{I\times J}, g\rangle|)_{I\times J}$ is equal to the norm of
$(\langle h_{I\times J}, g\rangle)_{I\times J}$.  See~\cite[Chapter 1]{lindenstrauss-tzafriri:1977}.

If $1 < p,p',q,q' < \infty$ and $\frac{1}{p} + \frac{1}{p'} = 1$, $\frac{1}{q} + \frac{1}{q'} = 1$,
it is recorded by Capon~\cite{capon:1982} that there is a constant $C_{p,q}$ such that for any
finite linear combination $f$ of Haar functions we have
\begin{equation*}
  C_{p,q}^{-1} \|f\|_{L^p(L^q)} \leq \|f\|_{H^p(H^q)} \leq C_{p,q} \|f\|_{L^p(L^q)}.
\end{equation*}
Consequently, the identity operator provides an isomorphism between $H^p(H^q)$ and $L^p(L^q)$, and
the dual of $H^p(H^q)$ identifies with $H^{p'}(H^{q'})$.  Capon's argument is based on the
observation by Pisier that the $\umd_p$ property of a Banach space does not depend on the value of
$1 < p < \infty$. For a proof of Pisier's observation, we refer to~\cite{maurey:1974-1975}
respectively~\cite[Chapter~5]{pisier:2016}.

For the limiting cases we have $H^1(H^q)^* = \bmo(H^{q'})$, $H^p(H^1)^* = H^{p'}(\bmo)$ and
$H^1(H^1)^* = \bmo(\bmo)$. See Maurey~\cite{maurey:1980} and M\"uller~\cite{mueller:1994}.

Let $\{\mathscr B_R : R\in \drec\}$ be a pairwise disjoint family, where each set~$\mathscr B_R$ is
a finite collection of disjoint dyadic rectangles. Given a vector of scalars
$\beta = (\beta_R : R\in \bigcup_{Q\in\drec} \mathscr{B}_Q)$, we define
\begin{equation}\label{eq:block-basis}
  b_R^{(\beta)}(x,y)
  = \sum_{Q\in \mathscr B_R} \beta_Q h_Q(x,y),
  \qquad x,y\in [0,1)
\end{equation}
and we call $\{b_R^{(\beta)} : R\in \drec\}$ the \emph{block basis generated by}
$\{\mathscr B_R : R\in \drec\}$ and $\beta = (\beta_R : R\in \bigcup_{Q\in\drec} \mathscr{B}_Q)$.
Now, let $1\leq p,q < \infty$ be fixed.  Note that $\{b_R^{(\beta)} : R\in \drec\}$ is
$1$-unconditional in $H^p(H^q)$ since $\{h_R : R\in\drec\}$ is $1$-unconditional in $H^p(H^q)$, i.e.
\begin{equation*}
  \Big\| \sum_{R\in\drec} \gamma_R \alpha_R b_R^{(\beta)}\Big\|_{H^p(H^q)}
  \leq \sup_{R\in\drec} |\gamma_R|\ \Big\| \sum_{R\in\drec} \alpha_R b_R^{(\beta)}\Big\|_{H^p(H^q)},
  \qquad (\gamma_R)_{R\in\drec}\in\ell^\infty(\drec),
\end{equation*}
whenever the series $\sum_{R\in\drec} \alpha_R b_R^{(\beta)}$ converges.  We say that the system
$\{b_R^{(\beta)} : R\in \drec\}$ is equivalent to the Haar system $\{h_R : R\in \drec\}$ if the
operator $B_\beta : H^p(H^q)\to H^p(H^q)$ given by
\begin{equation*}
  B_\beta(f)  = \sum_{R\in\drec} \frac{\langle f, h_R\rangle}{\|h_R\|_2^2}\, b_R^{(\beta)},
  \qquad f\in H^p(H^q),
\end{equation*}
is bounded and an isomorphism onto its range.  In this case, whenever
$C_1, C_2>0$ are constants such that  
\begin{equation*}
  \frac{1}{C_1} \| f\|_{H^p(H^q)}
  \leq \| B_\beta f\|_{H^p(H^q)}
  \leq C_2 \| f\|_{H^p(H^q)},
  \qquad f\in H^p(H^q),
\end{equation*}
we say that $\{b_R^{(\beta)} : R \in \drec\}$ is \emph{$C_1C_2$-equivalent} to
$\{h_R : R \in \drec\}$.

If $\beta_R = 1$ for each $R\in\drec$, then we write $b_R$ instead of $b_R^{(\beta)}$ and $B$ in
place of $B_\beta$.

\subsection{Uniform weak and weak* limits}\label{sec:uniformly-weak-weak}\hfill

\noindent
Let $\Gamma$ denote the closed unit ball of~$\ell^\infty(\mathscr{R})$, so that $\Gamma$ consists of
all families $\gamma = (\gamma_R : R\in\mathscr{R})$ of scalars with $|\gamma_R|\leq 1$ for each
$R\in\mathscr{R}$. Given $\gamma \in\Gamma$, the $1$-un\-conditionality of the bi-parameter Haar
system implies that the definition
\begin{equation}\label{eq:multiplier}
  M_\gamma\colon h_R\mapsto \gamma_R h_R,\qquad R\in\mathscr{R} 
\end{equation}
extends uniquely to an operator of norm~$\sup_R |\gamma_R|$ on~$H^p(H^q)$.

\begin{lem}\label{lem:thesequencefm}
 For  $m\in\mathbb N$, let $\mathscr{X}_m$ and $\mathscr{Y}_m$ be
  non-empty, finite families of pairwise disjoint dyadic intervals,
  define $f_m = \sum_{I\in\mathscr{X}_m,\, J\in\mathscr{Y}_m} h_{I\times
    J}$, $X_m = \bigcup\mathscr{X}_m$, and $Y_m =
  \bigcup\mathscr{Y}_m$, and let $1\leq p,q < \infty$.  Then:
  \begin{enumerate}[(i)]
  \item\label{lem:thesequencefm:0} $\|f_m\|_{H^p(H^q)} = |X_m|^{1/p}|Y_m|^{1/q}$ for all
    $m\in \mathbb{N}$;
  \item\label{lem:thesequencefm:1} $\|f_m\|_{H^p(H^q)^*} = |X_m|^{1-1/p}|Y_m|^{1-1/q}$ for all
    $m\in \mathbb{N}$.
  \end{enumerate}
 Suppose in addition that:
  \begin{itemize}
  \item $\mathscr{X}_m\cap\mathscr{X}_n=\emptyset$ or $\mathscr{Y}_m\cap\mathscr{Y}_n=\emptyset$
    whenever $m,n\in\mathbb N$ are distinct;
  \item $X_m = X_n$ and $Y_m = Y_n$ for all $m,n\in\mathbb{N}$.
  \end{itemize}
  Then:
  \begin{enumerate}[(i)]
    \setcounter{enumi}{2}
  \item\label{lem:thesequencefm:2} the sequence $(|X_m|^{-1/p}|Y_m|^{-1/q}f_m)_{m\in\mathbb N}$
    in~$H^p(H^q)$ is isometrically equivalent to the unit vector basis
    of~$\ell^2$;
  \item\label{lem:thesequencefm:3} for each $g\in H^p(H^q)^*$,
    $\sup_{\gamma\in\Gamma}|\langle M_\gamma f_m, g\rangle|\to 0$ as $m\to\infty$;
  \item\label{lem:thesequencefm:4} for each $g\in H^p(H^q)$,
    $\sup_{\gamma\in\Gamma}|\langle M_\gamma g, f_m\rangle|\to 0$ as $m\to\infty$.
  \end{enumerate}
\end{lem}

Note that in~\eqref{lem:thesequencefm:0},~\eqref{lem:thesequencefm:2},
and~\eqref{lem:thesequencefm:3}, we regard $f_m$ as an element of~$H^p(H^q)$, whereas
in~\eqref{lem:thesequencefm:1} and~\eqref{lem:thesequencefm:4}, we regard it as an element
of~$H^p(H^q)^*$.

\begin{proof} Set $\mathscr{B}_m = \{ I\times J : I\in\mathscr{X}_m,\, J\in\mathscr{Y}_m\}$ for each
  $m\in\mathbb N$.

  \eqref{lem:thesequencefm:0}. This follows immediately from the definition of $\|\cdot\|_{H^p(H^q)}$.

  \eqref{lem:thesequencefm:1}. For any
  $g = \sum_{K\times L\in \mathscr B_m} a_{K\times L}h_{K\times L}\in H^p(H^q)$ we obtain by
  H\"older's inequality that
  \begin{align*}
    |\langle f_m, g\rangle|
    & \leq \sum_{K\in \mathscr X_m} |K| \sum_{L\in \mathscr Y_m} |a_{K\times L}| |L|
      \leq |Y_m|^{1-1/q}\sum_{K\in \mathscr X_m} |K| \Bigl(
      \sum_{L\in \mathscr Y_m} |a_{K\times L}|^q |L|
      \Bigr)^{1/q}\\
    & \leq |X_m|^{1-1/p} |Y_m|^{1-1/q}\biggl(
      \sum_{K\in \mathscr X_m} |K| \Bigl(
      \sum_{L\in \mathscr Y_m} |a_{K\times L}|^q |L|
      \Bigr)^{p/q}
      \biggr)^{1/p}\\
    & = |X_m|^{1-1/p} |Y_m|^{1-1/q} \|g\|_{H^p(H^q)},
  \end{align*}
  and thus we have proved $\|f_m\|_{H^p(H^q)^*} \leq |X_m|^{1-1/p} |Y_m|^{1-1/q}$. For the other
  inequality, recall from~\eqref{lem:thesequencefm:0} that
  $\|f_m\|_{H^p(H^q)} = |X_m|^{1/p}|Y_m|^{1/q}$, thus
  \begin{equation*}
    \langle f_m, f_m\rangle
    = |X_m||Y_m|
    = |X_m|^{1-1/p}|Y_m|^{1-1/q} \|f_m\|_{H^p(H^q)}.
  \end{equation*}

  \eqref{lem:thesequencefm:2}. We observe that the first of the
  additional assumptions ensures that
  $\mathscr{B}_m\cap\mathscr{B}_n=\emptyset$ whenever $m,n\in\mathbb
  N$ are distinct. Set $X := X_m$ and $Y := Y_m$ for some (and hence
  all) $m\in\mathbb{N}$, and   let $(c_m)_{m\in\mathbb N}$ be a sequence
  of scalars that vanishes eventually. Since
  \begin{equation*}
    \sum_{R\in\mathscr{B}_m} \charfun_R(x,y) = \Bigl(\sum_{I\in\mathscr{X}_m}
    \charfun_I(x)\Bigr)\Bigl(\sum_{J\in\mathscr{Y}_m} \charfun_J(y)\Bigr) = \charfun_X(x)\charfun_Y(y)
  \end{equation*}  
  for all $m\in\mathbb N$ and $x,y\in[0,1)$,~\eqref{HpHqnorm} implies that
  \begin{align*} \Bigl\|\sum_m c_m f_m \Bigr\|_{H^p(H^q)}^p &= \int_0^1\biggl(\int_0^1\Bigl(\sum_m
    |c_m|^2 \charfun_X(x) \charfun_Y(y)\Bigr)^{q/2}\mathrm{d}y\biggr)^{p/q}\mathrm{d} x\\ &=
    \Bigl(\sum_m |c_m|^2\Bigr)^{p/2}|X|\,|Y|^{p/q},
  \end{align*}
  from which the conclusion follows.

  \eqref{lem:thesequencefm:3}. Let $g\in H^p(H^q)^*$ and $\varepsilon >0$. For each
  $R\in\mathscr{R}$, we can choose a scalar~$\beta_R$ with $|\beta_R| =1$ such that
  $\beta_R\langle h_R, g\rangle = |\langle h_R, g\rangle|$. Set $\beta = (\beta_R)\in\Gamma$.
  By~\eqref{lem:thesequencefm:2}, the sequence $(f_m)_{m\in\mathbb N}$ converges weakly to~$0$, so
  we can find $m_0\in\mathbb N$ such that $|\langle f_m, M_\beta^*g\rangle|\leq\varepsilon$ whenever
  $m\geq m_0$. Then, for each $\gamma = (\gamma_R)\in\Gamma$ and $m\geq m_0$ we have
  \begin{align*}
    |\langle M_\gamma f_m, g \rangle |
    &= \Bigl|\sum_{R\in\mathscr{B}_m} \gamma_R\langle h_R, g\rangle\Bigr|
      \leq \sum_{R\in\mathscr{B}_m} |\langle h_R, g\rangle|\\
    &= \sum_{R\in\mathscr{B}_m} \beta_R\langle h_R, g\rangle
      = \langle M_\beta f_m, g\rangle
      \leq\varepsilon,
  \end{align*}
  as required.

  \eqref{lem:thesequencefm:4}. Given $g\in H^p(H^q)$ and $\varepsilon >0$, we choose a finite
  subset~$\mathscr{F}$ of~$\mathscr{R}$ such that $\| g - Pg\|_{H^p(H^q)}\leq \varepsilon$, where
  $P : H^p(H^q)\to H^p(H^q)$ is the orthogonal projection given by
  $Pf = \sum_{R\in\mathscr{F}} \frac{\langle f, h_R\rangle}{|R|} h_R$.  Since the sets
  $\mathscr{B}_m,\, m\in\mathbb N$, are pairwise disjoint and~$\mathscr{F}$ is finite, we can find
  $m_0\in\mathbb N$ such that
  $\bigl(\bigcup_{m\geq m_0}\mathscr{B}_m\bigr)\cap\mathscr{F} =\emptyset$. Then, for each
  $m\geq m_0$ and $\gamma\in\Gamma$, we have $P^*f_m =0$, and hence
  \begin{align*}
    |\langle M_\gamma g, f_m\rangle|
    &= |\langle M_\gamma g, (I-P)^*f_m\rangle |= |\langle M_\gamma (I-P)g, f_m\rangle|\\
    &\leq \| M_\gamma\|\,\|g - Pg\|_{H^p(H^q)} \| f_m\|_{H^p(H^q)^*}
      \leq \varepsilon,
  \end{align*}  
  where we have used that $M_\gamma$ commutes with~$P$, and that
  \begin{equation*}
    \| f_m\|_{H^p(H^q)^*} = |X|^{1-1/p}|Y|^{1-1/q}\leq 1
  \end{equation*}
  by~\eqref{lem:thesequencefm:1}.
\end{proof}

\subsection{Embeddings and projections}\label{subsec:local_product}\hfill

\noindent
For each $R\in \drec$ let $\mathscr X_R, \mathscr Y_R\subset \dint$ denote non-empty, finite
collections of dyadic intervals that define the collection of dyadic rectangles $\mathscr B_R$ by
\begin{equation}\label{eq:symbols-0}
  \mathscr B_R
  = \{K\times L\, :\, K\in \mathscr X_R,\, L\in \mathscr Y_R\},
  \qquad R\in \drec.
\end{equation}
Now~\eqref{eq:block-basis} assumes the following form, if $\beta_R = 1$ for each $R\in\drec$:
\begin{equation}\label{eq:block-basis-capon}
  b_R(x,y)
  = \Big( \sum_{K\in \mathscr X_R} h_K(x) \Big) \Big(\sum_{L\in \mathscr Y_R} h_L(y)\Big),
  \qquad R\in \drec;
\end{equation}
see Figure~\ref{fig:2dhaar-layout}.
\begin{figure}[bt]
  \begin{center}
    \includegraphics[scale=0.25]{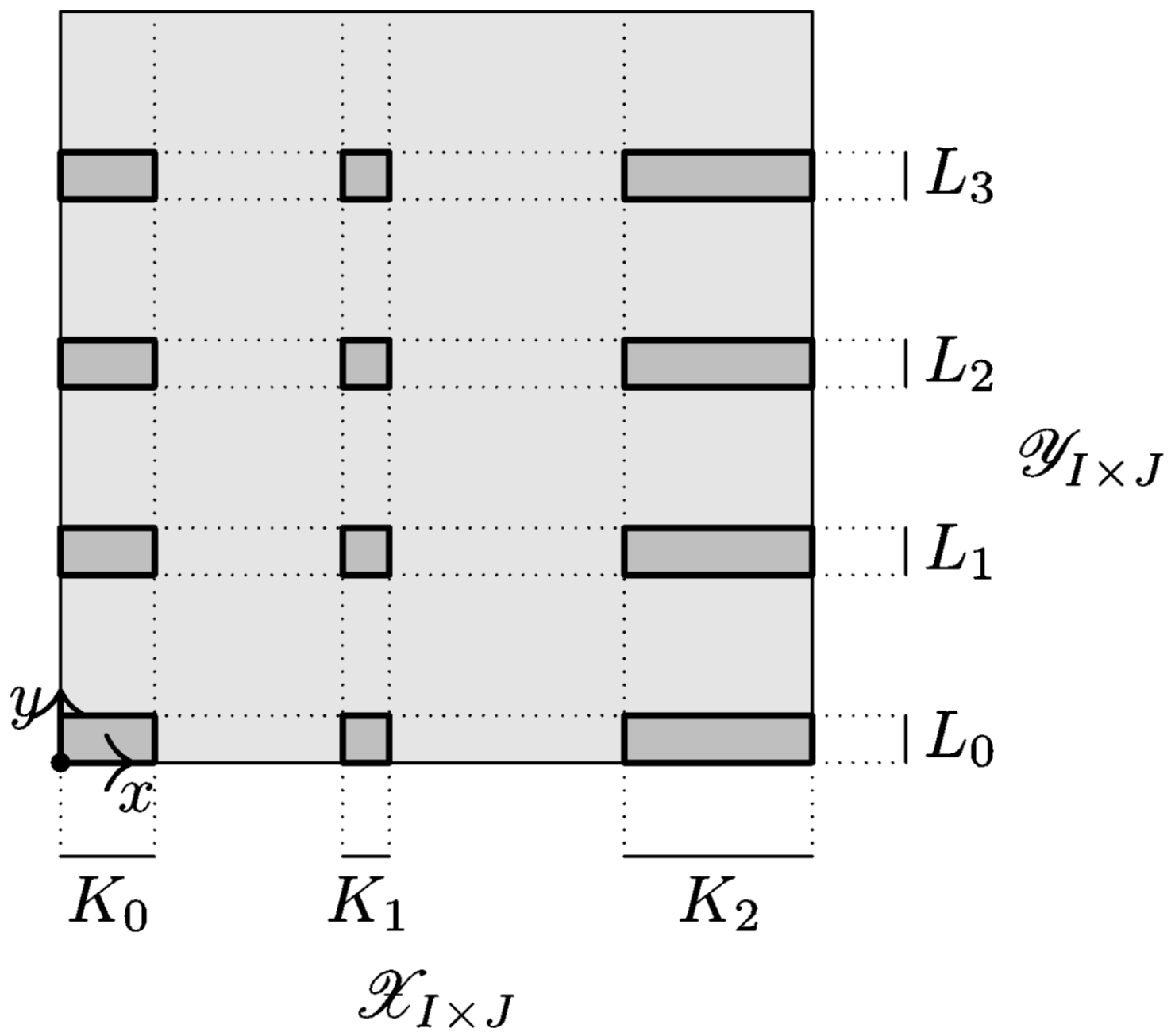}
  \end{center}
  \caption{For a dyadic rectangle $I\times J\in \drec$, this figure depicts
    $\mathscr B_{I\times J} = \mathscr X_{I\times J}\times \mathscr Y_{I\times J}$ (the collection
    of the dark gray rectangles) contained in the unit square (the light gray area). Here,
    $\mathscr X_{I\times J} = \{K_0,K_1,K_2\}$.  The dyadic rectangles in
    $K_i\times \mathscr Y_{I\times J}$ are connected by dotted lines.}\label{fig:2dhaar-layout}
\end{figure}

Capon~\cite{capon:1982} discovered a condition for $\{\mathscr B_R : R \in \drec\}$ which ensures
that the block basis $\{b_R : R \in \drec\}$ given by~\eqref{eq:block-basis-capon} is equivalent to
the Haar system $\{h_R : R \in \drec\}$ in $H^p(H^q)$, whenever $1 < p,q < \infty$ (see
Theorem~\ref{thm:capon}).  The local product condition~(P\ref{enu:p1})--(P\ref{enu:p4}) has its
roots in Capon's seminal work~\cite{capon:1982}.

We now introduce some notation. For $R \in\drec$ we set
\begin{equation}\label{eq:symbols-1}
  X_R = \bigcup\{ K : K\in \mathscr X_R \}
  \qquad\text{and}\qquad
  Y_R = \bigcup\{L : L\in \mathscr Y_R\}.
\end{equation}
For each $I_0\times J_0\in \drec$ we consider the following unions
\begin{equation}\label{eq:symbols-2}
  X_{I_0}
  = \bigcup\{ X_{I_0\times J} : J\in\dint\},
  \quad
  Y_{J_0}
  = \bigcup\{ Y_{I\times J_0} : I\in\dint\}.
\end{equation}
Clearly, for all $I\times J\in \drec$ the following crucial inclusions hold true:
\begin{equation}\label{eq:XY-inclusions}
  X_{I\times J} \subset X_I
  \qquad\text{and}\qquad
  Y_{I\times J} \subset Y_J.
\end{equation}

We say that $\{\mathscr B_{I\times J}: I\times J\in\drec\}$ given by~\eqref{eq:symbols-0} satisfies
the \emph{local product condition} with constants $C_X,C_Y>0$, if the following four
properties~(P\ref{enu:p1}), (P\ref{enu:p2}), (P\ref{enu:p3}) and~(P\ref{enu:p4}) hold true.
\begin{enumerate}[(P1)]
\item\label{enu:p1}%
  For all $R\in \drec$ the collection $\mathscr B_R$ consists of pairwise disjoint dyadic
  rectangles, and for all $R_0,R_1\in\drec$ with $R_0\neq R_1$ we have
  $\mathscr B_{R_0} \isect \mathscr B_{R_1} = \emptyset$.
\item\label{enu:p2}%
  For all $I\times J, I_0\times J_0, I_1\times J_1\in\drec$ with $I_0 \cap I_1 = \emptyset$,
  $I_0\union I_1\subset I$ and $J_0 \cap J_1 = \emptyset$, $J_0\union J_1\subset J$ we have
  \begin{align*}
    X_{I_0}\isect X_{I_1} &= \emptyset, &X_{I_0}\union X_{I_1} & \subset X_I,\\
    Y_{J_0}\isect Y_{J_1} &= \emptyset, &Y_{J_0}\union Y_{J_1} & \subset Y_J.
  \end{align*}
\item\label{enu:p3}%
  For each $R=I\times J\in\drec$, we have
  \begin{equation*}
    |I| \leq C_X |X_R|,
    \qquad |X_I| \leq C_X |I|,
    \qquad |J| \leq C_Y |Y_R|,
    \qquad |Y_J| \leq C_Y |J|.
  \end{equation*}
\item\label{enu:p4}%
  For all $I_0\times J_0, I\times J\in \drec$ with $I_0\times J_0\subset I\times J$ and for every
  $K\in \mathscr X_{I\times J}$ and $L\in \mathscr Y_{I\times J}$, we have
  \begin{equation*}
    \frac{|K\isect X_{I_0}|}{|K|} \geq  C_X^{-1}\frac{|X_{I_0}|}{|X_I|}
    \qquad\text{and}\qquad
    \frac{|L\isect Y_{J_0}|}{|L|} \geq  C_Y^{-1}\frac{|Y_{J_0}|}{|Y_J|}.
  \end{equation*}
\end{enumerate}
See Figure~\ref{fig:gamlen-gaudet} for the collections $\mathscr X_R$, $R\in \drec$, and
Figure~\ref{fig:2dhaar-one-step} as well as Figure~\ref{fig:2dhaar} for a depiction of
$\mathscr X_R$ and $\mathscr Y_R$, $R\in \drec$.

\begin{figure}[bt]
  \begin{center}
    \includegraphics{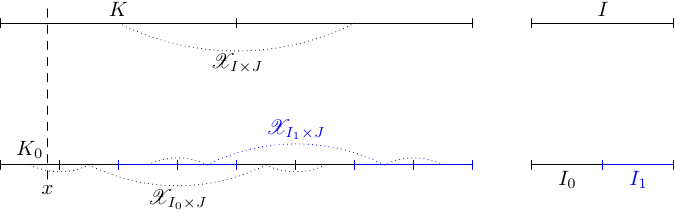}
  \end{center}
  \caption{The figure depicts the collections $\mathscr X_{I\times J}$, $\mathscr X_{I_0\times J}$,
    $\mathscr X_{I_1\times J}$, with $I_0\union I_1 = I$ and $I_0\cap I_1 = \emptyset$,
    $J\in \dint$. Given $x\in [0,1)$, the dashed vertical line connects the intervals $K_0$ and $K$
    with $x\in K_0 \subset K$. By~(P\ref{enu:p2}) we have $X_{I_0}\subset X_I$, and in the
    figure~(P\ref{enu:p4}) is realized by $\frac{|K\isect X_{I_0}|}{|K|} = \frac{|X_{I_0}|}{|X_I|}$.
  }\label{fig:gamlen-gaudet}
\end{figure}
\begin{figure}[bt]
  \begin{center}
    \includegraphics[scale=0.25]{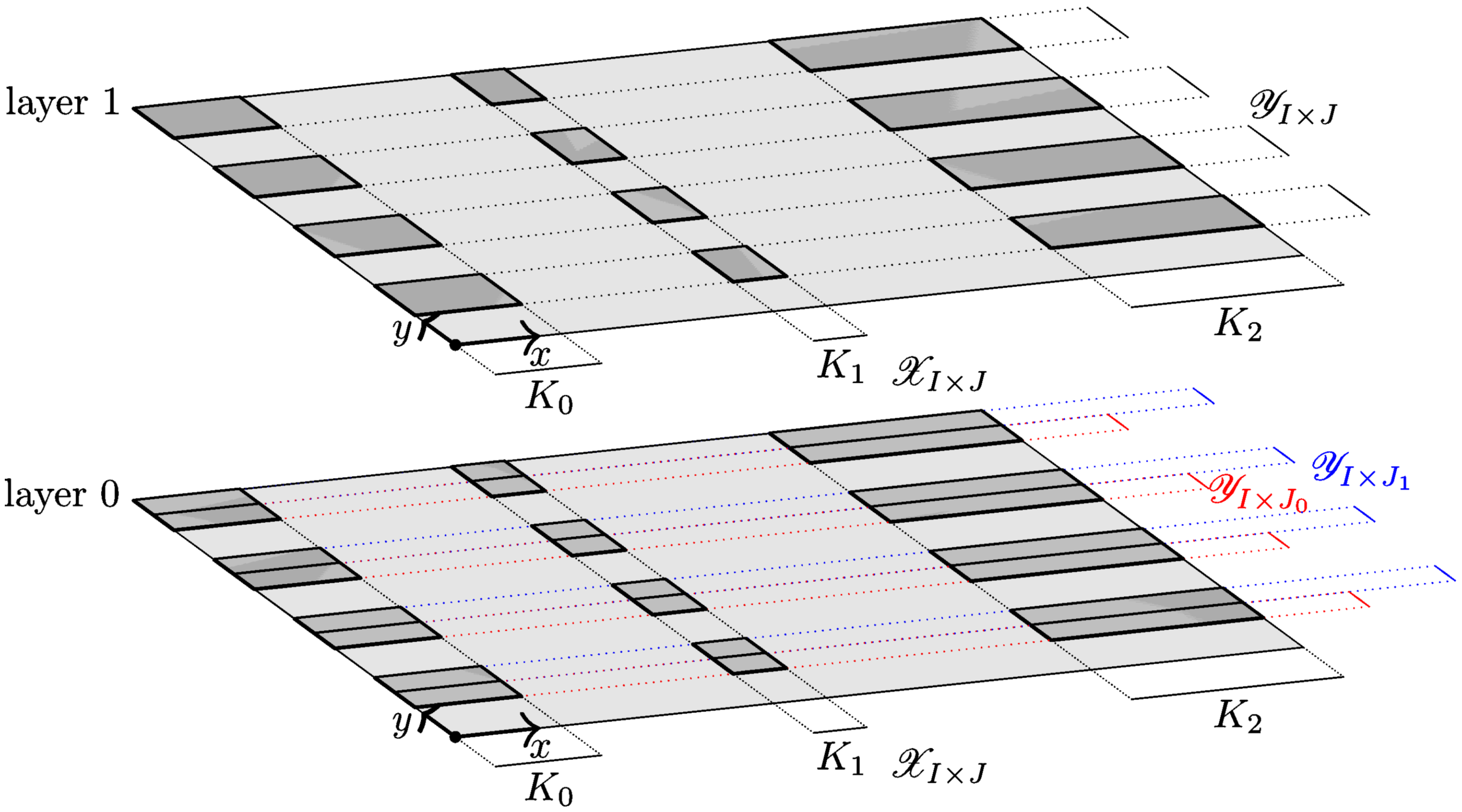}
  \end{center}
  \caption{The dyadic rectangles $I\times J$, $I\times J_0$ and $I\times J_1$ in $\drec$ are such
    that $J_0\union J_1 = J$ and $J_0\cap J_1 = \emptyset$. This figure depicts the collections
    $\mathscr B_{I\times J} = \mathscr X_{I\times J}\times \mathscr Y_{I\times J}$ in the top layer,
    and $\mathscr B_{I\times J_0} = \mathscr X_{I\times J_0}\times \mathscr Y_{I\times J_0}$ and
    $\mathscr B_{I\times J_1} = \mathscr X_{I\times J_1}\times \mathscr Y_{I\times J_1}$ in the
    bottom layer. Here,
    $\mathscr X_{I\times J} = \mathscr X_{I_0\times J} = \mathscr X_{I_1\times J} =
    \{K_0,K_1,K_2\}$. Each interval in $\mathscr Y_{I\times J}$ is split in two intervals, which are
    then placed into $\mathscr Y_{I\times J_0}$ and $\mathscr Y_{I\times J_1}$,
    respectively.}\label{fig:2dhaar-one-step}
\end{figure}
\begin{figure}[bt]
  \begin{center}
    \includegraphics[scale=0.25]{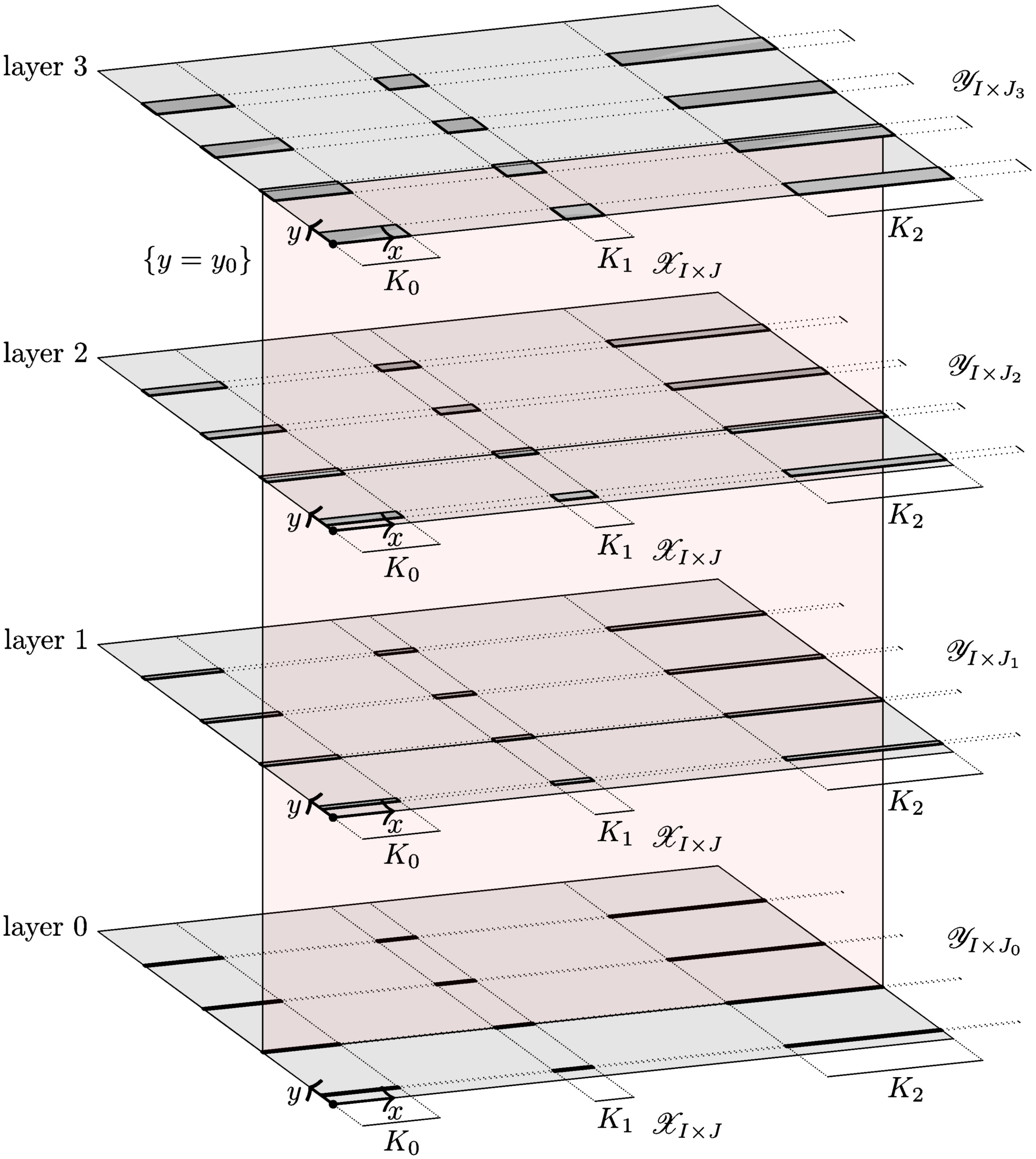}
  \end{center}
  \caption{In the figure, $\mathscr X_{I\times J_j} = \{K_0,K_1,K_2\}$, $0\leq j \leq 3$, whereas
    $\mathscr Y_{I\times J_j}$ changes with each layer $0\leq j\leq 3$.  For $y_0\in [0,1)$, the
    light red vertical plane connects the lines $\ell = \{(x,y_0) : x\in [0,1) \}$ in the four
    layers depicted in the figure.}\label{fig:2dhaar}
\end{figure}

\noindent
\begin{thm}[Capon]\label{thm:capon}
  Let $1 \leq p,q < \infty$.  If the conditions~(P\ref{enu:p1})--(P\ref{enu:p3}) are satisfied, then
  $\{b_{I\times J}\, :\, I\times J\in \drec\}$ is $C$-equivalent to
  $\{h_{I\times J}\, :\, I\times J\in \drec\}$ in $H^p(H^q)$, where $C$ depends only on $C_X$ and
  $C_Y$.
\end{thm}
We emphasize that $p$ or $q$ may take the value $1$ in the above theorem. By a duality argument,
M. Capon~\cite{capon:1982} showed the equivalence stated in Theorem~\ref{thm:capon} implies that the
orthogonal projection $P\, :\, H^p(H^q)\to H^p(H^q)$ given by
\begin{equation}\label{eq:ortho-proj}
  Pf
  = \sum_{I\times J\in \drec} \frac{\langle f, b_{I\times J}\rangle}{\|b_{I\times J}\|_2^2} b_{I\times J}
\end{equation}
is bounded on $H^p(H^q)$, whenever $1 < p,q < \infty$. We point out that the parameters $p=1$ or
$q=1$ are both excluded by the duality argument. Indeed, the duality argument of Capon shows that
\begin{equation*}
  \|P : H^p(H^q)\to H^p(H^q)\|\leq C(p,q,C_X,C_Y),
\end{equation*}
where the constants $C(p,q,C_X,C_Y)\to \infty$ in each of the cases $p\to 1$, $p\to \infty$,
$q\to 1$ or $q\to \infty$.

The next theorem is our first major step towards proving Theorem~\ref{thm:2d-andrew}. We show that
the operator $P$ is bounded on $H^p(H^q)$, $1\leq p,q < \infty$ with an upper estimate for the norm
independent of $p$ or $q$. Specifically, Theorem~\ref{thm:projection} includes the cases $p=1$ or
$q=1$.
\begin{thm}\label{thm:projection}
  Let $1 \leq p,q < \infty$, let $\{ \mathscr B_R\, :\, R\in \drec \}$ be a pairwise disjoint family
  which satisfies the local product condition~{\normalfont{(P\ref{enu:p1})--(P\ref{enu:p4})}} with
  constants~$C_X$ and~$C_Y$, and let $\beta = (\beta_Q : Q\in \bigcup_{R\in \drec} \mathscr B_R)$ be
  a family of scalars such that
  \begin{equation*}
    M:= \sup_Q |\beta_Q| < \infty.
  \end{equation*}
  Then the operators $B_\beta,A_\beta : H^p(H^q)\to H^p(H^q)$ given by
  \begin{equation*}
    B_\beta f = \sum_{R\in \drec} \frac{\langle f, h_R\rangle}{\|h_R\|_2^2} b_R^{(\beta)}
    \qquad\text{and}\qquad
    A_\beta f = \sum_{R\in \drec} \frac{\langle f, b_R^{(\beta)}\rangle}{\|b_R\|_2^2} h_R
  \end{equation*}
  satisfy the estimates
  \begin{equation}\label{thm:projection:estimates}
    \begin{aligned}
      \|B_\beta f \|_{H^p(H^q)} & \leq M C_X^{1/p}C_Y^{1/q} \|f\|_{H^p(H^q)},
      &f&\in H^p(H^q),\\
      \|A_\beta f \|_{H^p(H^q)} &\leq M C_X^{3+1/p} C_Y^{3+1/q} \|f\|_{H^p(H^q)}, &f&\in H^p(H^q).
    \end{aligned}
  \end{equation}
  If we additionally assume that
  \begin{equation*}
    m:= \inf_Q |\beta_Q| > 0,
  \end{equation*}
  and if we define the vector of scalars
  $\gamma = \bigl( \gamma_Q : Q\in \bigcup_{R\in \drec} \mathscr B_R \bigr)$ by
  $\beta_Q\gamma_Q = 1$, then the diagram
  \begin{equation}\label{thm:projection:diagram}
    \vcxymatrix{H^p(H^q) \ar[rr]^{I_{H^p(H^q)}} \ar[rd]_{B_\beta} & & H^p(H^q)\\
      &  H^p(H^q)  \ar[ru]_{A_\gamma} &
    }
  \end{equation}
  is commutative, and the operator $A_\gamma$ satisfies the estimate
  $\|A_\gamma\|\leq m^{-1} C_X^{3+1/p} C_Y^{3+1/q}$.  Moreover, the composition
  $P_{\beta,\gamma} = B_\beta A_\gamma$ is the projection $P_{\beta,\gamma} : H^p(H^q)\to H^p(H^q)$
  given by
  \begin{equation*}
    P_{\beta,\gamma}(f)
    = \sum_{R\in\drec} \frac{\langle f, b_R^{(\gamma)}\rangle}{\|b_R\|_2^2}
    b_R^{(\beta)}.
  \end{equation*}
  Consequently, the range of $B_\beta$ is complemented (by $P_{\beta,\gamma}$), and $B_\beta$ is an
  isomorphism onto its range.  Finally, if $\beta_Q=\gamma_Q=1$ for each $Q$, then
  $P_{\beta,\gamma}$ conincides with the orthogonal projection $P$ defined by~\eqref{eq:ortho-proj}.
\end{thm}

Before we proceed with the proof, we record some simple facts.
\begin{lem}\label{lem:projection-simple}
  Let $\mathscr B_R = \mathscr X_R\times \mathscr Y_R \subset \drec$, $R\in \drec$ satisfy the
  conditions~(P\ref{enu:p1}) and~(P\ref{enu:p3}). Then
  \begin{equation*}
    C_X^{-1}C_Y^{-1} |R| \leq
    \|b_R\|_2^2
    \leq C_X C_Y |R|,
    \qquad R\in \drec.
  \end{equation*}
\end{lem}

\begin{proof}
  Let $R\in \drec$ be fixed.  By condition~(P\ref{enu:p1}) and~\eqref{eq:symbols-0}, the collections
  $\mathscr X_R$ and $\mathscr Y_R$ each consist of pairwise disjoint dyadic intervals, thus,
  Lemma~\ref{lem:thesequencefm}~\eqref{lem:thesequencefm:0} yields
  \begin{equation*}
    \|b_R\|_2^2
    = |X_R| |Y_R|.
  \end{equation*}
  By~(P\ref{enu:p3}) and~\eqref{eq:XY-inclusions} we obtain
  \begin{equation*}
    C_X^{-1} C_Y^{-1} |R|
    \leq |X_R| |Y_R|
    \leq C_X C_Y |R|.
    \qedhere
  \end{equation*}

\end{proof}

Below we use Minkowski's inequality in various function spaces. For ease of reference, we include it
in the form that we need it.
\begin{lem}\label{lem:projection}
  Let $(\Omega,\mu)$ be a probability space.
  \begin{enumerate}[(i)]
  \item\label{enu:lem:projection-1} Let $1\leq r <\infty$ and let $g_k\in L^r(\Omega)$ be real
    valued. Then
    \begin{equation*}
      \int_\Omega \Big( \sum_k g_k^2 \Big)^{r/2} \dif \mu
      \geq \Big( \sum_k \big(\int_\Omega g_k\dif \mu\big)^2 \Big)^{r/2}.
    \end{equation*}
  \item\label{enu:lem:projection-2} Let $1\leq r,s < \infty$ and let $g_{k,\ell}\in L^s(\Omega)$ be
    real valued. Then
    \begin{equation*}
      \int_\Omega \Big(\sum_k\big(\sum_\ell g_{k,\ell}^2\big)^{s/2}\Big)^{r/s}\dif\mu
      \geq \bigg(\sum_k
      \Big(\sum_\ell \big(\int_\Omega g_{k,\ell}\dif\mu\big)^2\Big)^{s/2}
      \bigg)^{r/s}.
    \end{equation*}
  \end{enumerate}
\end{lem}

\begin{proof}
  First, we apply Minkowski's inequality (see
  e.g.~\cite[Corollary~5.4.2]{garling:2007},~\cite[Theorem~202]{hardy_littlewood_polya:1952}) to the
  integral and the sum over $\ell$:
  \begin{equation*}
    \bigg(\sum_k
    \Big(\sum_\ell \big(\int_\Omega g_{k,\ell}\dif\mu\big)^2\Big)^{s/2}
    \bigg)^{1/s}
    \leq \bigg(\sum_k
    \Big(\int_\Omega \big( \sum_\ell g_{k,\ell}^2\big)^{1/2}\dif\mu \Big)^{s}
    \bigg)^{1/s}.
  \end{equation*}
  Secondly, applying Minkowski's inequality to the integral and the sum over $k$ yields
  \begin{equation*}
    \bigg(\sum_k
    \Big(\int_\Omega \big( \sum_\ell g_{k,\ell}^2\big)^{1/2}\dif\mu \Big)^{s}
    \bigg)^{1/s}
    \leq \int_\Omega \Big(\sum_k
    \big( \sum_\ell g_{k,\ell}^2\big)^{s/2}
    \Big)^{1/s}\dif\mu.
  \end{equation*}
  Finally, we obtain~\eqref{enu:lem:projection-2} by H\"older's inequality.

  The assertion~\eqref{enu:lem:projection-1} follows from~\eqref{enu:lem:projection-2} by putting
  $s=2$.
\end{proof}

\begin{lem}\label{lem:reduction}
  Assume that $(Z_I : I\in \dint)$ satisfies the following condition: For all $I, I_0, I_1\in\dint$
  with $I_0 \cap I_1 = \emptyset$, $I_0\union I_1\subset I$ we have that
  \begin{eqnarray*}
    Z_{I_0}\isect Z_{I_1} = \emptyset
    \qquad\text{and}\qquad
    Z_{I_0}\union Z_{I_1}  \subset Z_I.
  \end{eqnarray*}
  Let $0 < r < \infty$, $N_0\in \mathbb N$ and $c_I \geq 0$ and define
  \begin{equation*}
    f(z) = \Big( \sum_{I\in \dint^{N_0}} c_I \charfun_{Z_I}(z) \Big)^r.
  \end{equation*}
  Then
  \begin{equation*}
    \widetilde c_I = \Big( \sum_{E\supset I} c_E \Big)^r - \Big( \sum_{E\supsetneq I} c_E \Big)^r
  \end{equation*}
  satisfies $\widetilde c_I\geq 0$ and we obtain the identity
  \begin{equation*}
    f(z) = \sum_{I\in \dint^{N_0}} \widetilde c_I \charfun_{Z_I}(z).
  \end{equation*}
\end{lem}

\begin{proof}
  Observe that by telescoping and the tree structure of the sets $(Z_I : I\in\dint)$ we have that
  \begin{equation*}
    \Big( \sum_{I\in \dint^{N_0}} c_I \charfun_{Z_I}(z) \Big)^r
    = \sum_{I\in \dint^{N_0}} \widetilde c_I \charfun_{Z_I}(z).
  \end{equation*}
  The fact that $\widetilde c_I \geq 0$ is self-evident.
\end{proof}

\begin{myproof}[Proof of Theorem~\ref{thm:projection}]
  The proof will be split into three parts. In the first part, we will give the estimate for
  $B_\beta$, and in the second part, we will establish the estimate for $A_\beta$.

  \begin{proofstep}[Part 1: The estimate for $B_\beta$]
    We emphasize that our proof of the estimate for $B_\beta$ only uses the
    conditions~(P\ref{enu:p1})--(P\ref{enu:p3}); specifically, we do not use~(P\ref{enu:p4}).

    For $N_0\in \mathbb N$ we define the collections of indices
    \begin{subequations}\label{eq:rect-coll}
      \begin{equation}
        \drec_{N_0} = \{I_0\times J_0\in \drec\, :\, I_0, J_0\in \dint_{N_0} \}
      \end{equation}
      and
      \begin{equation}
        \drec^{N_0} = \{I_0\times J_0\in \drec\, :\, I_0, J_0\in \dint^{N_0} \}.
      \end{equation}
    \end{subequations}

    Let us assume that
    \begin{equation*}
      f = \sum_{R\in \drec^{N_0}} a_R h_R.
    \end{equation*}
    Then by~(P\ref{enu:p1}) and~\eqref{eq:symbols-0} we find that
    \begin{equation*}
      \|B_\beta f\|_{H^p(H^q)}^p
      = \int_0^1 \bigg(
      \int_0^1 \Big(
      \sum_{R\in \drec^{N_0}} |a_R|^2 \sum_{Q\in\mathscr{B}_R}|\beta_Q|^2 \charfun_Q(x,y)
      \Big)^{q/2} \dif y
      \bigg)^{p/q} \dif x
      .
    \end{equation*}
    Recall that $|\beta_{I\times J}|\leq M$ and that by~\eqref{eq:XY-inclusions}
    $\charfun_{X_{I\times J}}(x) \charfun_{Y_{I\times J}}(y)\leq
    \charfun_{X_I}(x)\charfun_{Y_J}(y)$, so we note
    \begin{equation}\label{eq:thm:projection:B-estimate:1}
      \|B_\beta f\|_{H^p(H^q)}^p
      \leq M^p \int_0^1 \bigg(
      \int_0^1 \Big(
      \sum_{I\times J\in \drec^{N_0}} |a_{I\times J}|^2 \charfun_{X_I}(x)\charfun_{Y_J}(y)
      \Big)^{q/2} \dif y
      \bigg)^{p/q} \dif x
      .
    \end{equation}
    If we define
    $c_J(x) = \sum_{I\in \dint^{N_0}} |a_{I\times J}|^2
    \charfun_{X_I}(x)$,~\eqref{eq:thm:projection:B-estimate:1} reads
    \begin{equation}\label{eq:thm:projection:B-estimate:2}
      \|B_\beta f\|_{H^p(H^q)}^p
      \leq M^p \int_0^1 \bigg(
      \int_0^1 \Big(
      \sum_{J\in \dint^{N_0}} c_J(x) \charfun_{Y_J}(y)
      \Big)^{q/2} \dif y
      \bigg)^{p/q} \dif x
      .
    \end{equation}
    Lemma~\ref{lem:reduction} yields the following identity for the inner integrand
    of~\eqref{eq:thm:projection:B-estimate:2}:
    \begin{equation}\label{eq:thm:projection:B-estimate:3}
      \Big( \sum_{J\in \dint^{N_0}} c_J(x) \charfun_{Y_J}(y) \Big)^{q/2}
      = \sum_{J\in \dint^{N_0}} \widetilde c_J(x) \charfun_{Y_J}(y),
    \end{equation}
    where
    $\widetilde c_J(x) = \big( \sum_{J_1\supset J} c_{J_1}(x) \big)^{q/2} - \big(
    \sum_{J_1\supsetneq J} c_{J_1}(x) \big)^{q/2} \geq 0$.
    Integrating~\eqref{eq:thm:projection:B-estimate:3} with respect to $y$ and using that
    $|Y_J| \leq C_Y |J|$ by~(P\ref{enu:p3}), we have
    \begin{equation*}
      \int_0^1 \Big( \sum_{J\in \dint^{N_0}} c_J(x) \charfun_{Y_J}(y) \Big)^{q/2} \dif y
      \leq C_Y \sum_{J\in \dint^{N_0}} \widetilde c_J(x) |J|.
    \end{equation*}
    Combining the latter estimate with~\eqref{eq:thm:projection:B-estimate:2} yields
    \begin{equation}\label{eq:thm:projection:B-estimate:4}
      \|B_\beta f\|_{H^p(H^q)}^p
      \leq M^p C_Y^{p/q} \int_0^1 \bigg(
      \sum_{J\in \dint^{N_0}} \widetilde c_J(x) |J|
      \bigg)^{p/q} \dif x
      .
    \end{equation}  
    It remains to estimate
    $\int_0^1\big( \sum_{J\in \dint^{N_0}} \widetilde c_J(x) |J| \big)^{p/q} \dif x$ from above by a
    constant multiple of $\|f\|_{H^p(H^q)}^p$. Note that
    \begin{equation*}
      \begin{aligned}
        \big( \sum_{J_1\supset J} c_{J_1}(x) \big)^{q/2} & = \Big( \sum_{I\in \dint^{N_0}} d_{I,J}
        \charfun_{X_I}(x) \Big)^{q/2},
        \qquad\text{where}\ d_{I,J} = \sum_{J_1\supset J} |a_{I\times J_1}|^2,\\
        \big( \sum_{J_1\supsetneq J} c_{J_1}(x) \big)^{q/2} & = \Big( \sum_{I\in \dint^{N_0}}
        e_{I,J} \charfun_{X_I}(x) \Big)^{q/2}, \qquad\text{where}\ e_{I,J} = \sum_{J_1\supsetneq J}
        |a_{I\times J_1}|^2,
      \end{aligned}
    \end{equation*}
    and that $\widetilde c_J(x)$ was defined as the difference between the two quantities, above. By
    Lemma~\ref{lem:reduction}, we obtain
    \begin{equation*}
      \begin{aligned}
        \Big( \sum_{I\in \dint^{N_0}} d_{I,J} \charfun_{X_I}(x) \Big)^{q/2}
        & = \sum_{I\in \dint^{N_0}} \widetilde d_{I,J} \charfun_{X_I}(x),\\
        \Big( \sum_{I\in \dint^{N_0}} e_{I,J} \charfun_{X_I}(x) \Big)^{q/2} & = \sum_{I\in
          \dint^{N_0}} \widetilde e_{I,J} \charfun_{X_I}(x),
      \end{aligned}
    \end{equation*}
    where
    \begin{equation*}
      \begin{aligned}
        \widetilde d_{I,J} & = \Big( \sum_{I_1\supset I} d_{I_1,J} \Big)^{q/2}
        - \Big( \sum_{I_1\supsetneq I} d_{I_1,J} \Big)^{q/2} \geq 0,\\
        \widetilde e_{I,J} & = \Big( \sum_{I_1\supset I} e_{I_1,J} \Big)^{q/2} - \Big(
        \sum_{I_1\supsetneq I} e_{I_1,J} \Big)^{q/2} \geq 0.
      \end{aligned}
    \end{equation*}
    Summing up, in between~\eqref{eq:thm:projection:B-estimate:4} and here, we have shown that
    \begin{equation}\label{eq:thm:projection:B-estimate:5}
      \|B_\beta f\|_{H^p(H^q)}^p
      \leq M^p C_Y^{p/q} \int_0^1\Big( \sum_{I\in \dint^{N_0}} f_I \charfun_{X_I}(x) \Big)^{p/q}
      \dif x,
    \end{equation}
    where $f_I = \sum_{J\in \dint^{N_0}} |J| (\widetilde d_{I,J} - \widetilde e_{I,J})$.

    It is important to show that $f_I \geq 0$, for all $I\in \dint^{N_0}$. To this end, note the
    identity
    \begin{align*}
      \widetilde d_{I,J} - \widetilde e_{I,J}
      & = \Big( \sum_{\substack{I_1\supset I\\J_1\supset J}} |a_{I_1\times J_1}|^2 \Big)^{q/2}
      - \Big( \sum_{\substack{I_1\supsetneq I\\J_1\supset J}} |a_{I_1\times J_1}|^2 \Big)^{q/2}\\
      & \qquad - \Big( \sum_{\substack{I_1\supset I\\J_1\supsetneq J}} |a_{I_1\times J_1}|^2 \Big)^{q/2}
      + \Big( \sum_{\substack{I_1\supsetneq I\\J_1\supsetneq J}} |a_{I_1\times J_1}|^2 \Big)^{q/2}
      .
    \end{align*}
    Let $J_0\in \dint_{N_0}$, then grouping together the first with the third term as well as the
    second with the fourth, and summing the latter identity over $J\supset J_0$ yields
    \begin{equation*}
      \sum_{J\supset J_0} \widetilde d_{I,J} - \widetilde e_{I,J}
      = \Big( \sum_{\substack{I_1\supset I\\J_1\supset J_0}} |a_{I_1\times J_1}|^2 \Big)^{q/2}
      - \Big( \sum_{\substack{I_1\supsetneq I\\J_1\supset J_0}} |a_{I_1\times J_1}|^2 \Big)^{q/2}
      \geq 0.
    \end{equation*}
    Since we have
    \begin{equation*}
      f_I = \sum_{J_0\in \dint_{N_0}} |J_0| \sum_{J\supset J_0} (\widetilde d_{I,J} - \widetilde e_{I,J}),
    \end{equation*}
    we showed that $f_I \geq 0$.

    A final application of Lemma~\ref{lem:reduction} gives
    \begin{equation*}
      \int_0^1\Big( \sum_{I\in \dint^{N_0}} f_I \charfun_{X_I}(x) \Big)^{p/q} \dif x
      = \int_0^1 \sum_{I\in \dint^{N_0}} \widetilde f_I \charfun_{X_I}(x) \dif x
      = \sum_{I\in \dint^{N_0}} \widetilde f_I |X_I|,
    \end{equation*}
    where
    $\widetilde f_I = \big( \sum_{I_1\supset I} f_I \big)^{p/q} - \big( \sum_{I_1\supsetneq I} f_I
    \big)^{p/q}\geq 0$. Using~(P\ref{enu:p3}) in the above identity and combining it
    with~\eqref{eq:thm:projection:B-estimate:5} yields
    \begin{equation*}
      \|B_\beta f\|_{H^p(H^q)}^p
      \leq C_X M^p C_Y^{p/q} \sum_{I\in \dint^{N_0}} \widetilde f_I |I|.
    \end{equation*}
    Finally, we remark that
    \begin{equation*}
      \|f\|_{H^p(H^q)}^p = \sum_{I\in \dint^{N_0}} \widetilde f_I |I|.
    \end{equation*}
    To see this, it suffices to apply Lemma~\ref{lem:reduction} as above.
  \end{proofstep}

  \begin{proofstep}[Part 2: The estimate for $A_\beta$] 
    Let $N_0\in\mathbb{N}$, and define the collections of building blocks $\mathscr B_{N_0}$ and
    $\mathscr B^{N_0}$ by
    \begin{equation*}
      \mathscr B_{N_0}
      = \{ K_0\times L_0\in \mathscr B_{I_0\times J_0}\, :\, I_0\times J_0\in \drec_{N_0}\}
    \end{equation*}
    and
    \begin{equation*}
      \mathscr B^{N_0}
      = \{ K\times L\in \mathscr B_{I\times J}\, :\, I\times J\in \drec^{N_0}\},
    \end{equation*}
    where $\drec_{N_0}$ and $\drec^{N_0}$ are defined in~\eqref{eq:rect-coll}.  Taking into account
    that the bi-parameter Haar system is a $1$-unconditional basis of $H^p(H^q)$, it suffices to
    consider only those $f$ that can be written as follows:
    \begin{equation*}
      f = \sum_{K\times L\in \mathscr B^{N_0}} a_{K\times L} h_{K\times L}.
    \end{equation*}

    We will now estimate $\|A_\beta f\|_{H^p(H^q)}^p$. To this end, note that by the definitions of
    $A_\beta$ and the norm in $H^p(H^q)$ we have
    \begin{equation*}
      \|A_\beta f\|_{H^p(H^q)}^p
      = \int_0^1\bigg(
      \int_0^1 \Big(
      \sum_{R\in \drec^{N_0}}\frac{|\langle f, b_R^{(\beta)}\rangle|^2}{\|b_R\|_2^4} \charfun_R(x,y)
      \Big)^{q/2}
      \dif y
      \bigg)^{p/q}
      \dif x
      .
    \end{equation*}
    Since $\dint_{N_0}$ is a partition of the unit interval, we obtain that
    \begin{equation*}
      \|A_\beta f\|_{H^p(H^q)}^p
      = \sum_{I_0\in \dint_{N_0}} \int_{I_0}\bigg(
      \sum_{J_0\in \dint_{N_0}} \int_{J_0} \Big(
      \sum_{R\in \drec^{N_0}}\frac{|\langle f, b_R^{(\beta)}\rangle|^2}{\|b_R\|_2^4}
      \charfun_R(x,y)
      \Big)^{q/2}
      \dif y
      \bigg)^{p/q}
      \dif x.
    \end{equation*}
    Recall that $|\beta_Q|\leq M$, note that for $I_0,J_0\in \dint_{N_0}$ and $R\in \drec^{N_0}$ as
    in the above sums, $\charfun_R(x,y) = 1$ exactly when $R \supset I_0\times J_0$, and apply
    Lemma~\ref{lem:projection-simple} to obtain
    \begin{equation}\label{eq:projection_estimate}
      \begin{aligned}
        & \|A_\beta  f\|_{H^p(H^q)}^p\\
        & \leq M^p C_X^p C_Y^p \sum_{I_0\in \dint_{N_0}} |I_0| \bigg( \sum_{J_0\in \dint_{N_0}}
        |J_0| \Big( \sum_{\substack{R\in \drec^{N_0}\\R\supset I_0\times J_0}} \Big( \sum_{Q \in
          \mathscr B_R} \frac{|a_Q||Q|}{|R|} \Big)^2 \Big)^{q/2} \bigg)^{p/q} .
      \end{aligned}
    \end{equation}

    We continue by proving a lower bound for $\|f\|_{H^p(H^q)}^p$.  Set
    \begin{equation*}
      w_R = \sum_{Q\in \mathscr B_R} |a_Q| h_Q,
      \qquad R\in\drec^{N_0},
    \end{equation*}
    and observe that by~(P\ref{enu:p1}) we have
    \begin{equation*}
      \|f\|_{H^p(H^q)}^p
      = \int_0^1 \bigg(
      \int_0^1 \Big(
      \sum_{R\in \drec^{N_0}} w_R^2(x,y)
      \Big)^{q/2} \dif y
      \bigg)^{p/q} \dif x
      .
    \end{equation*}
    By~(P\ref{enu:p2}) the collections $\{X_{I_0} : I_0\in \dint_{N_0}\}$ and
    $\{Y_{J_0} : J_0\in \dint_{N_0}\}$ are each pairwise disjoint, thus we obtain
    \begin{equation*}
      \|f\|_{H^p(H^q)}^p
      \geq \sum_{I_0\in \dint_{N_0}} \int_{X_{I_0}} \bigg(
      \sum_{J_0\in \dint_{N_0}} |Y_{J_0}| \int_{Y_{J_0}} \Big(
      \sum_{R\in \drec^{N_0}} w_R^2(x,y)
      \Big)^{q/2} \frac{\dif y}{|Y_{J_0}|}
      \bigg)^{p/q} \dif x
      .
    \end{equation*}
    For fixed $I_0,J_0\in \dint_{N_0}$, $x\in X_{I_0}$, $y\in Y_{J_0}$ and $R\in \drec^{N_0}$, we
    have by~\eqref{eq:XY-inclusions} and~(P\ref{enu:p2}) that $w_R(x,y)\neq 0$ implies
    $R\supset I_0\times J_0$, so we obtain from the latter estimate together with~(P\ref{enu:p3})
    the following lower estimate for $C_Y^{p/q} \|f\|_{H^p(H^q)}^p$:
    \begin{equation}\label{eq:thm:projection:f-estimate:1}
      \sum_{I_0\in \dint_{N_0}} \int_{X_{I_0}} \bigg(
      \sum_{J_0\in \dint_{N_0}} |J_0| \int_{Y_{J_0}} \Big(
      \sum_{R\supset I_0\times J_0} w_R^2(x,y)
      \Big)^{q/2} \frac{\dif y}{|Y_{J_0}|}
      \bigg)^{p/q} \dif x
      .
    \end{equation}
    With $I_0, J_0\in \dint_{N_0}$ fixed, we now prepare for the application of
    Lemma~\ref{lem:projection} to the inner integral of the above estimate. We use the following
    specification. We put $\Omega = Y_{J_0}$, $\dif\mu = \frac{\dif y}{|Y_{J_0}|}$, and $r=q$.  In
    view of~\eqref{enu:lem:projection-1} of Lemma~\ref{lem:projection} we obtain that
    \begin{equation}\label{eq:thm:projection:f-estimate:2}
      \int_{Y_{J_0}} \Big(
      \sum_{R\supset I_0\times J_0} w_R^2(x,y)
      \Big)^{q/2} \frac{\dif y}{|Y_{J_0}|}
      \geq \Big(
      \sum_{R\supset I_0\times J_0}
      \Big(\int_{Y_{J_0}} |w_R(x,y)| \frac{\dif y}{|Y_{J_0}|} \Big)^2
      \Big)^{q/2}.
    \end{equation}
    By~(P\ref{enu:p1}) we have
    $|w_R(x,y)| = \sum_{K\times L\in \mathscr B_R} |a_{K\times L}| \charfun_K(x)\charfun_L(y)$,
    hence by~(P\ref{enu:p4}) and~(P\ref{enu:p3})
    \begin{align*}
      \int_{Y_{J_0}} |w_R(x,y)| \frac{\dif y}{|Y_{J_0}|}
      & = \sum_{K\times L\in \mathscr B_R}
        |a_{K\times L}| \frac{|L\isect Y_{J_0}|}{|Y_{J_0}|} \charfun_{K}(x)\\
      & \geq C_Y^{-2} \sum_{K\times L\in \mathscr B_R}
        |a_{K\times L}| \frac{|L|}{|J|} \charfun_{K}(x)
    \end{align*}
    for all $R\in \mathscr{R}^{N_0}$ with $R=I\times J\supset I_0\times J_0$. Combining the latter
    estimate with~\eqref{eq:thm:projection:f-estimate:2} and~\eqref{eq:thm:projection:f-estimate:1}
    we obtain the following lower estimate for $C_Y^{2p+p/q} \|f\|_{H^p(H^q)}^p$:
    \begin{equation}\label{eq:thm:projection:f-estimate:3}
      \sum_{I_0\in \dint_{N_0}} |X_{I_0}| \int_{X_{I_0}} \bigg(
      \sum_{J_0\in \dint_{N_0}} |J_0| \Big(
      \sum_{R\supset I_0\times J_0} v_R^2(x)
      \Big)^{q/2}
      \bigg)^{p/q} \frac{\dif x}{|X_{I_0}|}
      ,
    \end{equation}
    where we put
    $v_R(x) = \sum_{K\times L\in \mathscr B_R} \frac{|a_{K\times L}| |L|}{|J|} \charfun_K(x)$, if
    $R=I\times J$. With $I_0\in \dint_{N_0}$ fixed, we now prepare for the application of
    Lemma~\ref{lem:projection} to obtain a lower bound for the following term:
    \begin{equation}\label{eq:thm:projection:f-estimate:4}
      \int_{X_{I_0}} \bigg(
      \sum_{J_0\in \dint_{N_0}} |J_0| \Big(
      \sum_{R\supset I_0\times J_0} v_R^2(x)
      \Big)^{q/2}
      \bigg)^{p/q} \frac{\dif x}{|X_{I_0}|}.
    \end{equation}
    To this end, we use the following specification. We put $\Omega = X_{I_0}$,
    $\dif\mu = \frac{\dif x}{|X_{I_0}|}$, and $r=p$, $s=q$.  Invoking~\eqref{enu:lem:projection-2}
    of Lemma~\ref{lem:projection}, we find that~\eqref{eq:thm:projection:f-estimate:4} is bounded
    from below by
    \begin{equation}\label{eq:thm:projection:f-estimate:5}
      \bigg(
      \sum_{J_0\in \dint_{N_0}} |J_0| \Big(
      \sum_{R\supset I_0\times J_0} \Big(
      \int_{X_{I_0}} v_R(x) \frac{\dif x}{|X_{I_0}|}
      \Big)^2
      \Big)^{q/2}
      \bigg)^{p/q}.
    \end{equation}
    Recall that we defined
    $v_R(x) = \sum_{K\times L\in \mathscr B_R} \frac{|a_{K\times L}| |L|}{|J|} \charfun_K(x)$, if
    $R=I\times J$. By~(P\ref{enu:p4}) and~(P\ref{enu:p3}) we estimate
    \begin{align*}
      \int_{X_{I_0}} v_R(x) \frac{\dif x}{|X_{I_0}|}
      & = \sum_{K\times L\in \mathscr B_R} \frac{|a_{K\times L}| |L|}{|J|}
        \frac{|K\cap X_{I_0}|}{|X_{I_0}|}\\
      & \geq C_X^{-2}\sum_{Q\in \mathscr B_R} \frac{|a_Q| |Q|}{|R|}
    \end{align*}
    for all $R=I\times J\in \mathscr{R}^{N_0}$ with $R\supset I_0\times J_0$. Combining the latter
    estimate with~\eqref{eq:thm:projection:f-estimate:5},~\eqref{eq:thm:projection:f-estimate:4},
    and~\eqref{eq:thm:projection:f-estimate:3}, we obtain the following lower estimate for
    $C_X^{2p}C_Y^{2p+p/q} \|f\|_{H^p(H^q)}^p$:
    \begin{equation*}
      \sum_{I_0\in \dint_{N_0}} |X_{I_0}| \bigg(
      \sum_{J_0\in \dint_{N_0}} |J_0| \Big(
      \sum_{R\supset I_0\times J_0} \Big(
      \sum_{Q\in \mathscr B_R} \frac{|a_Q| |Q|}{|R|}
      \Big)^2
      \Big)^{q/2}
      \bigg)^{p/q}.
    \end{equation*}
    Finally, by~(P\ref{enu:p3}) the latter estimate yields
    \begin{equation}\label{eq:thm:projection:f-estimate:6}
      \begin{aligned}
        C_X^{2p+1} & C_Y^{2p+p/q} \|f\|_{H^p(H^q)}^p \geq\\
        & \sum_{I_0\in \dint_{N_0}} |I_0| \bigg( \sum_{J_0\in \dint_{N_0}} |J_0| \Big(
        \sum_{R\supset I_0\times J_0} \Big( \sum_{Q\in \mathscr B_R} \frac{|a_Q| |Q|}{|R|} \Big)^2
        \Big)^{q/2} \bigg)^{p/q}.
      \end{aligned}
    \end{equation}
    Direct comparison with~\eqref{eq:projection_estimate} gives
    \begin{equation*}
      \|A_\beta f\|_{H^p(H^q)}
      \leq M C_X^{3+1/p} C_Y^{3+1/q} \|f\|_{H^p(H^q)}.
    \end{equation*}

  \end{proofstep}
  \begin{proofstep}[Part 3: Conclusion of the proof]
    If additionally, we assume that $m:= \inf_Q |\beta_Q| > 0$, Part 2 implies that $A_\gamma$ is
    bounded by $ m^{-1} C_X^{3+1/p} C_Y^{3+1/q}$.  The commutativity of the
    diagram~\eqref{thm:projection:diagram} follows from the fact that $\beta_Q\gamma_Q = 1$.\qedhere
  \end{proofstep}
\end{myproof}

\subsection{A linear order on \pmb{$\drec$} and Capon's local product
  condition}\label{sec:order-clpc}\hfill

\noindent
In Section~\ref{sec:andrew}, we will iteratively construct collections of dyadic rectangles
$\mathscr{B}_R\subset\drec$, $R\in\drec$ satisfying Capon's local product condition.  This will be
accomplished by organizing the dyadic rectangles according to the linear order $\drless$ defined in
the present section, below.  The other purpose of this section is to introduce the auxiliary
condition (R\ref{enu:r1})--(R\ref{enu:r6}) and to show that it implies Capon's local product
condition (P\ref{enu:p1})--(P\ref{enu:p4}).

First, we define the bijective function $\mathcal{O}_{\mathbb N_0^2} : \mathbb N_0^2\to \mathbb N_0$
by
\begin{equation*}
  \mathcal O_{\mathbb N_0^2}(m,n) =
  \begin{cases}
    n^2 + m, & \text{if $m < n$},\\
    m^2 + m + n, & \text{if $m \geq n$}.
  \end{cases}
\end{equation*}
To see that $\mathcal O_{\mathbb N_0^2}$ is bijective consider that for each $k\in \mathbb N$:
\begin{itemize}
\item $\mathcal{O}_{\mathbb{N}_0^2}(0,0)=0$,
\item $m\mapsto \mathcal O_{\mathbb N_0^2}(m,k)$ maps $\{0,\dots,k-1\}$ bijectively onto
  $\{k^2,\dots,k^2+k-1\}$ and preserves the natural order on $\mathbb N_0$,
\item $\mathcal O_{\mathbb N_0^2}(k,0) = \mathcal O_{\mathbb N_0^2}(k-1,k)+1$,
\item $n\mapsto \mathcal O_{\mathbb N_0^2}(k,n)$ maps $\{0,\dots,k\}$ bijectively onto
  $\{k^2+k,\dots,k^2+2k\}$ and preserves the natural order on $\mathbb N_0$,
\item $\mathcal{O}_{\mathbb{N}_0^2}(0,k+1)=\mathcal{O}_{\mathbb{N}_0^2}(k,k)+1$.
\end{itemize}
See Figure~\ref{fig:ordering-relation:1} for a depiction of $\mathcal{O}_{\mathbb{N}_0^2}$.

Now, let $\lesslex$ denote the lexicographic order on $\mathbb R^3$. For two dyadic rectangles
$I_k\times J_k\in \drec$ with $|I_k|=2^{-m_k}$, $|J_k|=2^{-n_k}$, $k=0,1$, we define
$I_0\times J_0 \drless I_1\times J_1$ if and only if
\begin{equation*}
  \big( \mathcal O_{\mathbb N_0^2}(m_0,n_0),\inf I_0, \inf J_0 \big)
  \lesslex \big( \mathcal O_{\mathbb N_0^2}(m_1,n_1),\inf I_1, \inf J_1 \big).
\end{equation*}
Associated to the linear ordering $\drless$ is the bijective index function
$\drindex : \drec \rightarrow \mathbb N_0$ defined by
\begin{equation*}
  \drindex(R_0) < \drindex(R_1)
  \Leftrightarrow R_0 \drless R_1,
  \qquad R_0,R_1 \in \drec.
\end{equation*}
The geometry of a dyadic rectangle is linked to its index by the estimate
\begin{equation}\label{eq:ordering-1}
  (2^k -1)^2 \leq \drindex(I \times J) < (2^{k+1} -1)^2,
  \qquad \text{whenever $\min(|I|,|J|) = 2^{-k}$},
\end{equation}
and hence,
\begin{equation}\label{eq:ordering-estimate}
  \frac{1}{(1+\sqrt{i})^2}
  \leq |I|\, |J|,
  \qquad i=\drindex(I\times J).
\end{equation}
The index of a dyadic rectangle and its predecessors are related by
\begin{equation}\label{eq:ordering-2}
  \widetilde I\times J \drless I\times J,\ \text{for $I\neq [0,1)$}
  \qquad\text{and}\qquad
  I\times \widetilde J \drless I\times J,\ \text{for $J\neq [0,1)$},
\end{equation}
where we recall that for $I\ne [0,1)$, $\widetilde I$ is the unique dyadic interval satisfying
$\widetilde I \supset I$ and $|\widetilde I| = 2 |I|$.  See Figure~\ref{fig:ordering-relation:2} for
a picture of $\drindex$.
\begin{figure}[bt]
  \begin{center}
    \includegraphics{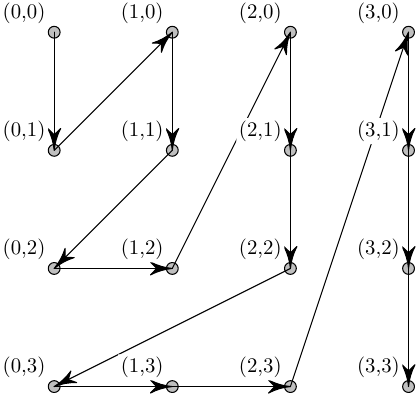}
  \end{center}
  \caption{This figure depicts the order of the first 16 pairs in $\mathbb N_0^2$ with respect to
    the map $\mathcal{O}_{\mathbb N_0^2}$.}\label{fig:ordering-relation:1}
\end{figure}
\begin{figure}[bt]
  \begin{center}
    \includegraphics{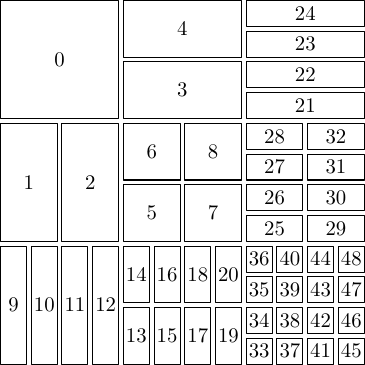}
  \end{center}
  \caption{The first $49$ rectangles and their indices $\drindex$.}\label{fig:ordering-relation:2}
\end{figure}

For a dyadic interval~$I$, we write $I^\ell$ and $I^r$ for the dyadic intervals which are the left
and right halves of $I$, respectively.  In the following definition, we use the notation introduced
in~\eqref{eq:symbols-1}, so that for a collection $\mathscr{X}_R$ (respectively, $\mathscr{Y}_R$) of
dyadic intervals, $X_R$ (respectively, $Y_R$) denotes its union.
\begin{dfn}\label{dfn:aux-fin}
  Let $\mathscr{A} = \drec$ or $\mathscr{A} = \{ R\in\drec : R \drlesseq R_0\}$ for some
  $R_0\in\mathscr{R}$.  We say that $\{ \mathscr{B}_R : R\in \mathscr{A}\}$ satisfies the
  \emph{auxiliary condition (R\ref{enu:r1})--(R\ref{enu:r6})} if the following properties hold true.
  \begin{enumerate}[(R1)]
  \item\label{enu:r1} For each $R\in\mathscr{A}$, there are non-negative integers $\mu(R)$, $\nu(R)$
    and non-empty sets $\mathscr{X}_R\subset\mathscr{D}_{\mu(R)}$ and
    $\mathscr{Y}_R\subset\mathscr{D}_{\nu(R)}$ such that
    $\mathscr{B}_R = \{ K\times L : K\in\mathscr{X}_R,\, L\in\mathscr{Y}_R\}$.

  \item\label{enu:r2} $\mu([0,1)\times [0,1)) = \nu([0,1)\times [0,1)) = 0$ and
    $\mathscr{X}_{[0,1)\times [0,1)} = \mathscr{Y}_{[0,1)\times [0,1)} = \{[0,1)\}$.

  \item\label{enu:r3} For each $I\in\mathscr{D}\setminus\{[0,1)\}$ with
    $R = I\times [0,1)\in\mathscr{A}$
    \begin{equation*}
      X_{I\times [0,1)} =
      \begin{cases} \bigcup\{ K^\ell :
        K\in\mathscr{D}_{\kappa(R)},\,
        K\subset X_{\widetilde{I}\times [0,1)}\} &\text{if}\ \ I = \widetilde{I}^\ell,\\
        \bigcup\{ K^r : K\in\mathscr{D}_{\kappa(R)},\, K\subset X_{\widetilde{I}\times [0,1)}\}
        &\text{if}\ \ I = \widetilde{I}^r,
      \end{cases}
    \end{equation*}
    where $\kappa(R) = \max\{\mu(S) : S\drless [0,|I|)\times [0,1)\}$;

  \item\label{enu:r4} If $R = I\times J\in\mathscr{A}$ with $|I|<|J|$, then
    \begin{equation*}
      \mu(R)> \max\{ \mu(S) : S\drless R\},
    \end{equation*}
    $X_R = X_{I\times [0,1)}$, and $\mathscr{Y}_{R} = \mathscr{Y}_{I'\times J}$, where
    $I'\in\mathscr{D}$ is the unique dyadic interval such that $I'\supset I$ and $|I'| = |J|$.
    
  \item\label{enu:r5} For $J\in\mathscr{D}\setminus\{[0,1)\}$ with $R = [0,1)\times J\in\mathscr{A}$
    \begin{equation*}
      Y_{[0,1)\times J} =
      \begin{cases} \bigcup\{ L^\ell :
        L\in\mathscr{D}_{\lambda(R)},\, L\subset Y_{[0,1)\times \widetilde{J}}\}
        &\text{if}\ \ J = \widetilde{J}^\ell,\\
        \bigcup\{ L^r : L\in\mathscr{D}_{\lambda(R)},\, L\subset
        Y_{[0,1)\times\widetilde{J}}\}
        &\text{if}\ \ J = \widetilde{J}^r,
      \end{cases}
    \end{equation*} where $\lambda(R) =
    \max\{\nu(S) : S\drless [0,1)\times [0,|J|)\}$.
    
  \item\label{enu:r6} If $R = I\times J\in\mathscr{A}\setminus\{[0,1)\times [0,1)\}$ with
    $|I|\ge|J|$, then
    \begin{equation*}
      \nu(R)> \max\{ \nu(S) : S\drless R\},
    \end{equation*}
    $Y_R = Y_{[0,1)\times J}$, and $\mathscr{X}_{R} = \mathscr{X}_{I\times J'}$, where
    $J'\in\mathscr{D}$ is the unique dyadic interval such that $J'\supset J$ and $|J'| = 2|I|$ if
    $I\ne[0,1)$, and $J'=[0,1)$ if $I= [0,1)$.
  \end{enumerate}
\end{dfn}

\begin{rem}\label{rem:aux-condition}
  Let $\{\mathscr{B}_R : R\in\mathscr{R}\}$ be a collection such that each of the finite
  sub-collections $\{\mathscr{B}_R : R\drlesseq R_0\}$, $R_0\in\drec$, satisfies the auxiliary
  condition (R\ref{enu:r1})--(R\ref{enu:r6}).  Then it is easy to see that
  $\{\mathscr{B}_R : R\in\drec\}$ itself satisfies the auxiliary condition
  (R\ref{enu:r1})--(R\ref{enu:r6}).
\end{rem}

\begin{lem}\label{conditionsensuringP1P4}
  Let $\{\mathscr{B}_R : R\in\mathscr{R}\}$ satisfy the auxiliary condition
  (R\ref{enu:r1})--(R\ref{enu:r6}). Then $\{\mathscr{B}_R : R\in\mathscr{R}\}$ satisfies the local
  product condition (P\ref{enu:p1})--(P\ref{enu:p4}) with constants $C_X = C_Y=1$.
\end{lem}

\begin{proof}
  The usual linear order $\prec$ on dyadic intervals is given by $ I_1\prec I_0$ if and only if
  either $|I_1|>|I_0|$ or $|I_1|=|I_0|$ and $\min I_1<\min I_0$.  The proof uses induction with
  respect to the linear orders $\prec$ and $\drless$.
  \begin{proofcase}[Verification of (P\ref{enu:p1})]
    For each $R\in\mathscr{R}$, $\mathscr{X}_R$ consists of pairwise disjoint intervals
    because~$\mathscr{X}_R$ is contained in~$\mathscr{D}_{\mu(R)}$. Similarly,
    $\mathscr{Y}_R\subset \dint_{\nu(R)}$ and consists of pairwise disjoint intervals, and therefore
    the rectangles in~$\mathscr{B}_R$ are pairwise disjoint.

    Now suppose that $R_0,R_1\in\mathscr{R}$ are distinct. By relabelling them if necessary, we may
    suppose that $R_1\drless R_0$, where $R_0 = I_0\times J_0\ne [0,1)\times [0,1)$.  To establish
    the disjointness of~$\mathscr{B}_{R_0}$ and~$\mathscr{B}_{R_1}$, we must show that either
    $\mathscr{X}_{R_0}$ and~$\mathscr{X}_{R_1}$ are disjoint or $\mathscr{Y}_{R_0}$
    and~$\mathscr{Y}_{R_1}$ are disjoint.  If $|I_0|<|J_0|$, then~(R\ref{enu:r4}) implies
    that~$\mu(R_0) > \mu(R_1)$, so that
    $\mathscr{X}_{R_0}\cap\mathscr{X}_{R_1}\subset\mathscr{D}_{\mu(R_0)}\cap\mathscr{D}_{\mu(R_1)} =
    \emptyset$.  Otherwise $|I_0|\ge|J_0|$, in which case a similar argument based
    on~(R\ref{enu:r6}) shows that $\mathscr{Y}_{R_0}\cap\mathscr{Y}_{R_1}= \emptyset$.
  \end{proofcase}

  \begin{proofcase}[Verification of (P\ref{enu:p2})]
    We begin by oberving that (R\ref{enu:r4}) and (R\ref{enu:r6}) imply that the sets $X_R$, $Y_R$,
    $X_I$, and~$Y_J$ defined in~\eqref{eq:symbols-1}--\eqref{eq:symbols-2} are given by
    \begin{equation}\label{XIandYJsets}
      X_R = X_{I\times [0,1)} = X_I\qquad\text{and}\qquad Y_R =
      Y_{[0,1)\times J} = Y_J,\qquad R =
      I\times J\in\mathscr{R}.
    \end{equation}
    Since the order $\prec$ is linear, and the set~$\mathscr{D}$ is countable and has a minimum
    element~$[0,1)$ with respect to~$\prec$, we may use induction on $I_0\in\mathscr{D}$ to prove
    the following two statements:
    \begin{enumerate}[(a)]
    \item\label{conditionsensuringP1P4eqP2i}
      $X_{I_0\times [0,1)}\cap X_{I_1\times [0,1)} = \emptyset$ and
      $Y_{[0,1)\times I_0}\cap Y_{[0,1)\times I_1} = \emptyset$ for each $I_1\in\mathscr{D}$ with
      $I_1\prec I_0$ and $I_0\cap I_1 = \emptyset$;
    \item\label{conditionsensuringP1P4eqP2ii} $X_{I_0\times [0,1)}\subset X_{I_1\times [0,1)}$ and
      $Y_{[0,1)\times I_0}\subset Y_{[0,1)\times I_1}$ for each $I_1\in\mathscr{D}$ with
      $I_0\subset I_1$.
    \end{enumerate}
    The statements~\eqref{conditionsensuringP1P4eqP2i} and~\eqref{conditionsensuringP1P4eqP2ii}
    above together with~\eqref{XIandYJsets} imply (P\ref{enu:p2}).  The start of the induction is
    easy. Indeed, suppose that $I_0 = [0,1)$. Then no $I_1$ satisfies $I_1\prec [0,1)$, so
    that~\eqref{conditionsensuringP1P4eqP2i} is vacuous, while~\eqref{conditionsensuringP1P4eqP2ii}
    holds trivially because $I_1 = [0,1)$ is the only dyadic interval which contains~$[0,1)$.

    Now let $I_0\in\mathscr{D}\setminus\{[0,1)\}$, and assume inductively
    that~\eqref{conditionsensuringP1P4eqP2i}--\eqref{conditionsensuringP1P4eqP2ii} have been
    established for each $I_0'\prec I_0$ (that
    is,~\eqref{conditionsensuringP1P4eqP2i}--\eqref{conditionsensuringP1P4eqP2ii} hold
    whenever~$I_0$ is replaced with~$I_0'$).  We shall prove the statements
    concerning~$X_{I_0\times [0,1)}$; the proofs for $Y_{[0,1)\times I_0}$ are similar, requiring
    only minor adjustments of the notation.

    To verify~\eqref{conditionsensuringP1P4eqP2i}, suppose that $I_1\in\mathscr{D}$ satisfies
    $I_1\prec I_0$ and $I_0\cap I_1 = \emptyset$. Then either $I_1\cap\widetilde{I}_0 = \emptyset$,
    or $I_1 = (\widetilde{I}_0)^\ell$ and $I_0 = (\widetilde{I}_0)^r$. (Note that because
    $I_1\prec I_0$, we cannot have $I_1 = (\widetilde{I}_0)^r$ and $I_0 = (\widetilde{I}_0)^\ell$.)
    In the first case, since $I_1\prec I_0$ and $\widetilde{I}_0\prec I_0$, the induction hypothesis
    implies that $X_{\widetilde{I}_0\times [0,1)}\cap X_{I_1\times [0,1)}=\emptyset$, from which the
    result follows because $X_{I_0\times [0,1)}\subset X_{\widetilde{I}_0\times [0,1)}$ by
    (R\ref{enu:r3}).

    In the second case, we observe that $\widetilde{I}_0 = \widetilde{I}_1$ and $|I_0| = |I_1|$, so
    that $\kappa(I_0\times [0,1)) = \kappa(I_1\times [0,1))$.  This implies that
    $X_{I_0\times [0,1)}$ and $X_{I_1\times [0,1)}$ are disjoint because $X_{I_0\times [0,1)}$ is
    the disjoint union of the \textsl{right} halves of the intervals
    $K\in\mathscr{D}_{\kappa(I_0\times [0,1))}$ with $K\subset X_{\widetilde{I}_0\times [0,1)}$,
    while $X_{I_1\times [0,1)}$ is the disjoint union of the \textsl{left} halves of the same
    intervals.

    Next, to prove~\eqref{conditionsensuringP1P4eqP2ii}, suppose that $I_1\in\mathscr{D}$ with
    $I_0\subset I_1$. The inclusion is obvious if $I_0=I_1$, so we may suppose that
    $I_0\subsetneq I_1$. Then we have $\widetilde{I}_0\subset I_1$, so the induction hypothesis
    implies that $X_{\widetilde{I}_0\times [0,1)}\subset X_{I_1\times [0,1)}$. Hence the statement
    follows from the fact that $X_{I_0\times [0,1)}\subset X_{\widetilde{I}_0\times [0,1)}$.
  \end{proofcase}

  \begin{proofcase}[Verification of (P\ref{enu:p3})]
    The proofs of (P\ref{enu:p3}) and (P\ref{enu:p4}) both rely on the following two identities:
    \begin{equation}\label{eqKcapXIandLcapYJ}
      |K\cap X_{I\times [0,1)}|
      = \frac{|K\cap  X_{\widetilde{I}\times [0,1)}|}{2}
      \qquad\text{and}\qquad 
      |L\cap Y_{[0,1)\times J}|
      = \frac{|L\cap Y_{[0,1)\times \widetilde{J}}|}{2},
    \end{equation}
    valid for $I,J\in\mathscr{D}\setminus\{[0,1)\}$, $K\in\mathscr{D}^{\kappa(I\times [0,1))}$, and
    $L\in\mathscr{D}^{\lambda([0,1)\times J)}$.

    We shall establish the first of these identities; again, the proof of the other requires only
    notational changes. For $I\in\mathscr{D}\setminus\{[0,1)\}$ and
    $K\in\mathscr{D}^{\kappa(I\times [0,1))}$, set
    $\mathscr{V}_I(K) = \{ K_0\in\mathscr{D}_{\kappa(I\times [0,1))} : K_0\subset K\cap
    X_{\widetilde{I}\times [0,1)}\}$. We claim that
    \begin{equation}\label{XItildeEq}
      K\cap  X_{\widetilde{I}\times [0,1)} = \bigcup\mathscr{V}_I(K)\quad\text{and}\quad
      K\cap X_{I\times [0,1)} = \begin{cases} \bigcup\{ K_0^\ell : K_0\in\mathscr{V}_I(K)\} &\text{if}\ I = \widetilde{I}^\ell\\
        \bigcup\{ K_0^r : K_0\in\mathscr{V}_I(K)\}  &\text{if}\ I = \widetilde{I}^r. \end{cases}
    \end{equation}
    Indeed, the inclusion $\bigcup\mathscr{V}_I(K)\subset K\cap X_{\widetilde{I}\times [0,1)}$ is
    clear from the definition of~$\mathscr{V}_I(K)$.  Conversely, for each
    $x\in K\cap X_{\widetilde{I}\times [0,1)}$, there is a (necessarily unique) interval
    $K_0\in\mathscr{X}_{\widetilde{I}\times [0,1)}$ such that $x\in K_0$. We have
    $\mu(\widetilde{I}\times [0,1))\le \kappa(I\times [0,1))$ because
    $\widetilde{I}\times [0,1)\drless [0,|I|)\times [0,1)$, so we can find
    $K_1\in\mathscr{D}_{\kappa(I\times [0,1))}$ such that $x\in K_1\subset K_0$. The sets $K_1$
    and~$K$ are not disjoint as they both contain~$x$; combined with the fact that $|K_1|\le |K|$,
    this shows that $K_1\subset K$. Moreover, we have
    $K_1\subset K_0\subset X_{\widetilde{I}\times [0,1)}$, so that $K_1\in\mathscr{V}_I(K)$, and
    hence $x\in K_1\subset\bigcup\mathscr{V}_I(K)$.

    Moving on to the second part of~\eqref{XItildeEq}, we obtain the inclusion~$\supset$ directly
    from the definition of~$\mathscr{V}_I(K)$ and~(R\ref{enu:r3}).  Conversely, suppose that
    $x\in K\cap X_{I\times [0,1)}$, so that $x\in K$ and either $x\in K_0^\ell$ or $x\in K_0^r$
    (depending on whether $I = (\widetilde{I})^\ell$ or $I = (\widetilde{I})^r$) for some
    $K_0\in\mathscr{D}_{\kappa(I\times [0,1))}$ with $K_0\subset X_{\widetilde{I}\times [0,1)}$. In
    both cases, we see that $K\cap K_0\ne\emptyset$ and $|K_0|\le |K|$, so that $K_0\subset K$, and
    hence $K_0\in\mathscr{V}_I(K)$, from which the inclusion follows.

    The first equation in~\eqref{eqKcapXIandLcapYJ} is immediate from~\eqref{XItildeEq} because
    $\mathscr{V}_I(K)$ consists of disjoint sets and $|K_0^\ell| = |K_0^r| = |K_0|/2$.

    We can now easily establish (P\ref{enu:p3}) with $C_X=C_Y=1$. By~\eqref{XIandYJsets}, we must
    show that
    \begin{equation}\label{eqP3}
      |X_{I\times [0,1)}| = |I|\qquad\text{and}\qquad
      |Y_{[0,1)\times I}| = |I|,\qquad I\in\mathscr{D}. 
    \end{equation}
    We do so by induction on~$I$. The start of the induction, where $I = [0,1)$, follows immediately
    from the fact that $X_{[0,1)\times [0,1)} = Y_{[0,1)\times [0,1)} = [0,1)$ by (R\ref{enu:r2}).

    Now let $I\in\mathscr{D}\setminus\{[0,1)\}$, and assume inductively that the result is true for
    each $I'\prec I$.  Using~\eqref{eqKcapXIandLcapYJ} with $K = L= [0,1)$, we obtain that
    $|X_{I\times [0,1)}| = |X_{\widetilde{I}\times [0,1)}|/2 = |\widetilde{I}|/2 = |I|$ because
    $\widetilde{I}\prec I$ and likewise $|Y_{[0,1)\times I}| = |I|$.
  \end{proofcase}

  \begin{proofcase}[Verification of (P\ref{enu:p4})]
    We shall prove that, for each $R_0 = I_0\times J_0$ and $R = I\times J$ in~$\mathscr{R}$ with
    $R_0\subset R$,
    \begin{equation}\label{eqP4}
      \frac{|K\cap X_{I_0\times [0,1)}|}{|I_0|} = \frac{|K|}{|I|}\qquad\text{and}\qquad
      \frac{|L\cap Y_{[0,1)\times J_0}|}{|J_0|} = \frac{|L|}{|J|},\qquad
      K\in\mathscr{X}_R,\,L\in\mathscr{Y}_R.
    \end{equation}
    By~\eqref{XIandYJsets} and~\eqref{eqP3}, this will verify (P\ref{enu:p4}) with $C_X = C_Y =1$.

    The proof of~\eqref{eqP4} is by induction on~$R_0$. The start of the induction is trivial
    because the only $R\in\mathscr{R}$ that contains $R_0 = [0,1)\times [0,1)$ is~$R_0$ itself.

    Now let $R_0\in\mathscr{R}\setminus\{[0,1)\times [0,1)\}$, and assume inductively
    that~\eqref{eqP4} has been verified for each $R_0'\drless R_0$. This time, we shall focus on the
    proof of the second identity in~\eqref{eqP4}; the proof of the first identity is similar, but
    formally slightly easier due to the lack of symmetry between conditions~(R\ref{enu:r4}) and
    (R\ref{enu:r6}): when $|I|=|J|$, we re-use an existing set as $\mathscr{X}_R$ and define a new
    set~$\mathscr{Y}_R$.

    Suppose that $R = I\times J\in\mathscr{R}$ with $R_0\subset R$, and let $L\in\mathscr{Y}_R$. If
    $J_0 = J$, then $L\subset Y_{[0,1)\times J_0}$, and the identity is immediate. Hence we may
    suppose that $J_0\subsetneq J$. Moreover, we may suppose that $|I|\ge |J|$. Indeed, if not, then
    by~(R\ref{enu:r4}) $\mathscr{Y}_R = \mathscr{Y}_{I'\times J}$, where $I'\in\mathscr{D}$
    satisfies $I'\supset I$ and $|I'| = |J|$, so that we may replace~$I$ with~$I'$ to obtain that
    $|I|\ge |J|$.

    Then we have $|J_0| < |J| = \min\{|I|,|J|\}$, so that $R\drless [0,1)\times [0,|J_0|)$, and
    hence $\lambda([0,1)\times J_0)\ge \nu(R)$; thus
    $L\in\mathscr{Y}_R\subset\mathscr{D}_{\nu(R)}\subset\mathscr{D}^{\lambda([0,1)\times J_0)}$, so
    that~\eqref{eqKcapXIandLcapYJ} shows that
    $|L\cap Y_{[0,1)\times J_0}| = |L\cap Y_{[0,1)\times \widetilde{J}_0}|/2$. Now
    $R_0' = I_0\times\widetilde{J}_0$ satisfies $R_0'\drless R_0$ and $R_0'\subset R$, and therefore
    the induction hypothesis implies that
    $|L\cap Y_{[0,1)\times \widetilde{J}_0}|/|\widetilde{J}_0| = |L|/|J|$. Hence the conclusion
    follows because $|\widetilde{J}_0| = 2|J_0|$.\qedhere
  \end{proofcase}
\end{proof}

Having obtained Theorem~\ref{thm:projection} and Lemma~\ref{conditionsensuringP1P4}, we are finally
prepared to prove Theorem~\ref{thm:2d-andrew}.

\section{Proof of Theorem~\ref{thm:2d-andrew}}\label{sec:andrew}

\noindent
Here, we prove that the identity operator on $H^p(H^q)$ factors through any operator
$T : H^p(H^q)\to H^p(H^q)$ having large diagonal with respect to the bi-parameter Haar system (see
Theorem~\ref{thm:2d-andrew}).  The basic pattern of our argument below is the following: we
carefully construct $\{\mathscr{B}_R : R\in\drec\}$ satisfying the auxiliary condition
(R\ref{enu:r1})--(R\ref{enu:r6}) (see Section~\ref{sec:projections}).  Moreover, these collections
are chosen in such a way that we are able to find signs $\varepsilon_Q\in\{\pm 1\}$,
$Q\in\bigcup_{R\in\drec}\mathscr{B}_R$, for which the block basis
$b_R^{(\varepsilon)} = \sum_{Q\in\mathscr{B}_R}\varepsilon_Q h_Q$, $R\in\drec$ has the following
properties: $|\langle T b_{R_1}^{(\varepsilon)}, b_{R_2}^{(\varepsilon)}\rangle|$ is small in the
precise sense of~\eqref{eq:induction-properties:b} below whenever
$R_1,R_2\in\drec$ are distinct,
and
\begin{equation*}
  |\langle T b_R^{(\varepsilon)}, b_R^{(\varepsilon)} \rangle|
  \geq \delta \|b_R^{(\varepsilon)}\|_2^2,
  \qquad R\in\drec.
\end{equation*}
Thereafter we apply the two main results of the preceding section, Theorem~\ref{thm:projection} and
Lemma~\ref{conditionsensuringP1P4}, and finally we construct a factorization of the identity
operator through $T$.

\begin{proof}[Proof of Theorem~\ref{thm:2d-andrew}]
  Let $1 \leq p,q < \infty$ and $\delta > 0$, and let $T : H^p(H^q)\to H^p(H^q)$ be an operator such
  that
  \begin{equation}\label{eq:large-diagonal-0}
    |\langle T h_R, h_R \rangle| \geq \delta |R|,
    \qquad R\in \drec.
  \end{equation}

  We define $\gamma = (\gamma_R : R\in\drec)$ by
  \begin{equation*}
    \gamma_R
    = \frac{\overline{\langle T h_R, h_R\rangle}}{|\langle T h_R, h_R\rangle|},
    \qquad R\in \drec.
  \end{equation*}
  Recall that in~\eqref{eq:multiplier} we defined the Haar multiplier $M_\gamma$ which satisfies
  $\|M_\gamma\| = 1$, and $\langle (TM_\gamma) h_R, h_R \rangle \geq \delta |R|$. Thereby, replacing
  $T$ with $TM_\gamma$, it suffices to consider the special case where
  \begin{equation}\label{eq:large-diagonal-1}
    \langle T h_R, h_R \rangle \geq \delta |R|,
    \qquad R\in \drec.
  \end{equation}

  \begin{proofstep}[Overview]
    Let $0 < \eta\le 1$. The main part of the proof consists of choosing collections of dyadic
    rectangles $\mathscr B_{R}$, $R \in \drec$ and suitable signs $\varepsilon=(\varepsilon_{Q})$
    such that $b_{R}^{(\varepsilon)} = \sum_{Q\in \mathscr B_{R}} \varepsilon_{Q} h_{Q}$ satisfies
    the following:
    \begin{itemize}
    \item The closed linear span of $\{b_R^{(\varepsilon)} : R\in\drec\}$ is complemented and
      isomorphic to $H^p(H^q)$.
    \item There is an operator $U : H^p(H^q) \to H^p(H^q)$ given by
      \begin{equation*}
        U(f) = \sum_{R\in\drec}
        \frac{\langle f, b_R^{(\varepsilon)}\rangle}
        {\langle Tb_R^{(\varepsilon)}, b_R^{(\varepsilon)}\rangle}
        b_R^{(\varepsilon)}.
      \end{equation*}
    \item For every finite linear combination $g = \sum_{R\in\drec} \lambda_R b_R^{(\varepsilon)}$
      we have
      \begin{equation*}
        \|UTg - g\|_{H^p(H^q)} \leq \frac{\eta}{2} \|g\|_{H^p(H^q)}.
      \end{equation*}
    \end{itemize}
  \end{proofstep}

  \begin{proofstep}[Preparation]
    Given $R = I\times J \in \drec$ we write
    \begin{subequations}\label{eq:decomp}
      \begin{equation}
        T h_{R} = \alpha_{R} h_{R} + r_{R},
      \end{equation}
      where
      \begin{equation}
        \alpha_{R} = \frac{\langle T h_{R}, h_{R} \rangle}{|R|}
        \qquad\text{and}\qquad
        r_{R} = \sum_{S\neq R}
        \frac{\langle T h_{R}, h_{S} \rangle}{|S|} h_{S}.
      \end{equation}
    \end{subequations}
    We note the estimates
    \begin{equation}\label{eq:a-estimate}
      \delta \leq \alpha_{R} \leq \|T\|
      \qquad\text{and}\qquad
      \|r_{R}\|_{H^p(H^q)} \leq 2\|T\| |I|^{1/p} |J|^{1/q}.
    \end{equation}
  \end{proofstep}

  \begin{proofstep}[Inductive construction of $b_{R}^{(\varepsilon)}$]
    We will now inductively define the block basis $\{ b_{R}^{(\varepsilon)}\, :\, R \in \drec \}$.
    For fixed $R\in \drec$, the block basis element $b_{R}^{(\varepsilon)}$ is determined by a
    collection of dyadic rectangles $\mathscr B_{R}\subset \drec$ and a suitable choice of signs
    $\varepsilon=(\varepsilon_{Q})$ and is of the following form:
    \begin{equation}\label{eq:block_basis}
      b_{R}^{(\varepsilon)} = \sum_{Q\in \mathscr B_{R}} \varepsilon_{Q} h_{Q}.
    \end{equation}
    From now on, we systematically use the following rule: whenever $\drindex(R) = i$ we set
    \begin{equation*}
      \mathscr B_i = \mathscr B_{R},
      \qquad
      b_i^{(\varepsilon)} = b_{R}^{(\varepsilon)},
      \qquad
      h_i = h_R.
    \end{equation*}
    We will construct collections $\{\mathscr B_i : i\in\mathbb{N}_0\}$ satisfying the auxiliary
    condition~(R\ref{enu:r1})--(R\ref{enu:r6}) and choose signs $\varepsilon=(\varepsilon_{Q})$ such
    that
    \begin{subequations}\label{eq:induction-properties}
      \begin{align}
        \sum_{j=0}^{i-1} |\langle T b_j^{(\varepsilon)}, b_i^{(\varepsilon)}\rangle|
        + |\langle b_i^{(\varepsilon)}, T^*b_j^{(\varepsilon)}\rangle|
        & \leq \eta\delta 4^{-i-2},\qquad i\in\mathbb{N},
          \label{eq:induction-properties:b}\\
        |\langle T b_i^{(\varepsilon)}, b_i^{(\varepsilon)} \rangle|
        & \geq \delta \|b_i^{(\varepsilon)}\|_2^2,\qquad i\in\mathbb{N}_0.
          \label{eq:induction-properties:c}
      \end{align}
    \end{subequations}

    The induction begins by putting
    \begin{equation}\label{eq:induction-init}
      \mathscr B_0 = \{[0,1)\times [0,1)\}
      \qquad\text{and}\qquad
      b_0^{(\varepsilon)} = h_{[0,1)\times [0,1)}.
    \end{equation}
    Consequently, $\mathscr{X}_{[0,1)\times [0,1)} = \mathscr{Y}_{[0,1)\times [0,1)} = \{[0,1)\}$
    and $\mu([0,1)\times [0,1)) = 0$, $\nu([0,1)\times [0,1)) = 0$.  Obviously, $\{\mathscr{B}_0\}$
    satisfies~(R\ref{enu:r1})--(R\ref{enu:r6}).

    Let $i_0\in \mathbb N$. At this stage we assume that
    \begin{itemize}
    \item $\{\mathscr B_j : 0\leq j\leq i_0-1\}$ satisfies the auxiliary
      condition~(R\ref{enu:r1})--(R\ref{enu:r6}).

    \item the block basis $\{b_j^{(\varepsilon)} : 0\leq j\leq i_0-1\}$ given
      by~\eqref{eq:block_basis} satisfies~\eqref{eq:induction-properties} (for $0\leq i\leq i_0-1$).
    \end{itemize}
    Now, we turn to the construction of $\mathscr B_{i_0}$ and $\varepsilon_{Q}$, where
    $Q\in \mathscr B_{i_0}$. In the first step we will find $\mathscr B_{i_0}$
    in~\eqref{eq:induction-step:block_basis_collection:dfn}, and only then we will choose the signs
    $\varepsilon_Q$, $Q\in\mathscr{B}_{i_0}$ in~\eqref{eq:diagonal_estimate-1}.  The collection
    $\mathscr B_{i_0}$ and the signs $\varepsilon_{Q}$, $Q\in\mathscr{B}_{i_0}$ then determine
    $b_{i_0}^{(\varepsilon)}$.
  \end{proofstep}

  \begin{proofstep}[Construction of $\mathscr B_{i_0}$]
    Let $I_0\times J_0\in \drec$ be such that $\drindex(I_0\times J_0) = i_0$.  We distinguish
    between the four cases
    \begin{equation*}
      |I_0|<|J_0|, J_0=[0,1),
      \qquad |I_0|<|J_0|, J_0\neq[0,1),
    \end{equation*}
    and
    \begin{equation*}
      |I_0|\geq |J_0|, I_0=[0,1),
      \qquad |I_0|\geq |J_0|, I_0\neq[0,1).
    \end{equation*}

    \noindent
    \begin{proofcase}[Case~1: $|I_0| < |J_0|$]
      Here, we will construct the collection $\mathscr B_{I_0\times J_0}$, for which the index
      rectangle $I_0\times J_0$ is ``below the diagonal''.

      First, we define
      \begin{equation}\label{eq:induct-case-1:mu-nu:2}
        \nu(I_0\times J_0) = \nu(I_0'\times J_0)
        \qquad\text{and}\qquad
        \mathscr Y_{I_0\times J_0} = \mathscr Y_{I_0'\times J_0},
      \end{equation}
      where $I_0'\in \dint$ is the unique interval such that $I_0'\supset I_0$ and $|I_0'|=|J_0|$.
      We remark that $\mu(I_0\times J_0)$ will be defined at the end of the proof
      in~\eqref{eq:induction-step:block_basis_collection:a}.

      \noindent
      \begin{minipage}[H]{.75\textwidth}
        \textsc{Case~1.a: $J_0 = [0,1)$.}  Here, we know that $I_0\neq [0,1)$. Recall that
        $\widetilde I_0$ denotes the dyadic predecessor of $I_0$, and note that
        $\mathscr B_{\widetilde I_0\times [0,1)}$ has already been defined.  The collections indexed
        by the black rectangles have already been constructed.  Here, we determine the collections
        for the gray rectangles.  The white ones will be treated later.
      \end{minipage}
      \begin{minipage}[H]{.25\textwidth}
        \begin{center}
          \includegraphics[scale=0.4]{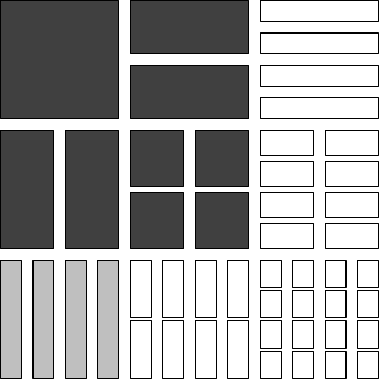}
        \end{center}
      \end{minipage}

      \bigskip\noindent
      Note that $[0,|I_0|)\times [0,1)\drlesseq I_0\times [0,1)$, and define the integer
      $\kappa(I_0\times [0,1))$ by
      \begin{equation*}
        \kappa(I_0\times [0,1))
        = \max\{\mu(Q) : Q\drless [0,|I_0|)\times [0,1)\}.
      \end{equation*}
      Recall that for a dyadic interval $K_0$ we denote its left half by $K_0^\ell$ and its right
      half by $K_0^r$.  Following the basic construction of Gamlen-Gaudet~\cite{gamlen_gaudet:1973},
      we proceed as follows.  The set $X_{\widetilde{I_0}\times [0,1)}$ has already been defined in
      a previous step of the construction.  Now we put
      \begin{equation*}
        X_{I_0\times [0,1)} =
        \begin{cases}
          \bigcup\{ K_0^\ell : K_0\in\mathscr{D}_{\kappa(I_0\times[0,1))},\,
          K_0\subset X_{\widetilde{I_0}\times [0,1)}\} &\text{if}\ \ I_0 = \widetilde{I_0}^\ell,\\
          \bigcup\{ K_0^r : K_0\in\mathscr{D}_{\kappa(I_0\times[0,1))},\, K_0\subset
          X_{\widetilde{I_0}\times [0,1)}\} &\text{if}\ \ I_0 = \widetilde{I_0}^r.
        \end{cases}
      \end{equation*}
      To finish the construction in Case~1.\textsc{a}, we define the family of high frequency covers
      of the set $X_{I_0\times [0,1)}\times [0,1)$ by putting
      \begin{equation}\label{eq:rectangles-case_1}
        \mathscr F_m = \{K\times [0,1)\in \drec\, :\, K\in\dint_m,\, K\subset X_{I_0\times[0,1)}\},
      \end{equation}
      for all $m > \kappa(I_0\times[0,1))$, see Figure~\ref{fig:building_blocks-1}, and observe that
      \begin{equation}\label{eq:rectangles-case_1a:pointset}
        \bigcup \mathscr{F}_m = X_{I_0\times [0,1)}\times [0,1).
      \end{equation}

      \begin{figure}[H]
        \begin{center}
          \includegraphics{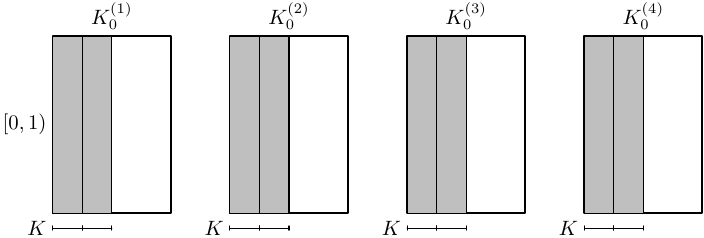}
        \end{center}
        \caption{The above figure depicts an instance of $\mathscr F_m$ in
          Case~1.\textsc{a}. $K_0^{(k)}$ is a dyadic interval such that
          $K_0^{(k)}\times [0,1)\in \mathscr B_{\widetilde I_0\times [0,1)}$, and $K$ is a dyadic
          interval such that $K\times [0,1)\in \mathscr F_m$.}\label{fig:building_blocks-1}
      \end{figure}
      \noindent
      \begin{minipage}[H]{.75\textwidth}
        \textsc{Case~1.b: $J_0 \neq [0,1)$.}  The
        collections indexed by the black rectangles have already been constructed. Here, we
        determine the collections for the gray rectangles. The white ones will be treated later.
      \end{minipage}
      \begin{minipage}[H]{.25\textwidth}
        \begin{center}
          \includegraphics[scale=0.4]{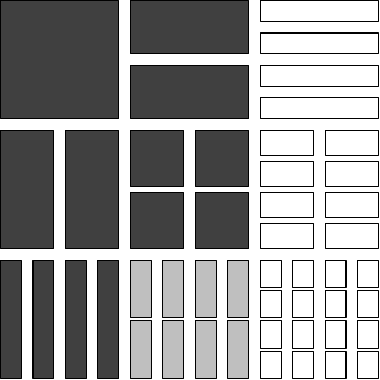}
        \end{center}
      \end{minipage}

      \bigskip\noindent By our induction hypothesis,
      $\{\mathscr{B}_{I\times J} : \drindex(I\times J) \leq i_0-1\}$
      satisfies~(R\ref{enu:r1})--(R\ref{enu:r6}).  Note that the set $X_{I_0\times [0,1)}$ is
      already defined.  To conclude the construction in Case~1.\textsc{b}, we define the high
      frequency covers of $X_{I_0\times [0,1)}\times Y_{I_0\times J_0}$ by
      \begin{equation}\label{eq:rectangles-case_1b}
        \mathscr F_m = \{K\times L_0\in \drec\, :\,
        K\in\dint_m,\, K\subset X_{I_0\times [0,1)},\ L_0\in \mathscr Y_{I_0\times J_0}
        \},
      \end{equation}
      for $m > \mu({I_0\times \widetilde J_0})$.
    \end{proofcase}

    \begin{proofcase}[Case~2: $|I_0|\geq |J_0|$]
      In this case, we will construct the collection $\mathscr B_{I_0\times J_0}$, for which the
      index rectangle $I_0\times J_0$ is ``on or above the diagonal''.

      First, we set
      \begin{equation}\label{eq:induct-case-2:mu-nu:2}
        \mu(I_0\times J_0) = \mu(I_0\times J_0')
        \qquad\text{and}\qquad
        \mathscr X_{I_0\times J_0} = \mathscr X_{I_0\times J_0'},
      \end{equation}
      where $J_0'\in\mathscr{D}$ is the unique dyadic interval such that $J_0'\supset J_0$ and
      $|J_0'| = 2|I|$ if $I\ne[0,1)$, and $J_0'=[0,1)$ if $I= [0,1)$.  We remark that
      $\nu(I_0\times J_0)$ will be defined at the end of the proof
      in~\eqref{eq:induction-step:block_basis_collection:b}.

      \noindent
      \begin{minipage}[H]{.75\textwidth}
        \textsc{Case~2.a: $I_0 = [0,1)$.} Note that $J_0\ne [0,1)$ and
        $\mathscr B_{[0,1)\times \widetilde J_0}$ has already been constructed. The collections
        indexed by the black rectangles have already been constructed. Here, we determine the
        collections for the gray rectangles. The white ones will be treated later.
      \end{minipage}
      \begin{minipage}[H]{.25\textwidth}
        \begin{center}
          \includegraphics[scale=0.4]{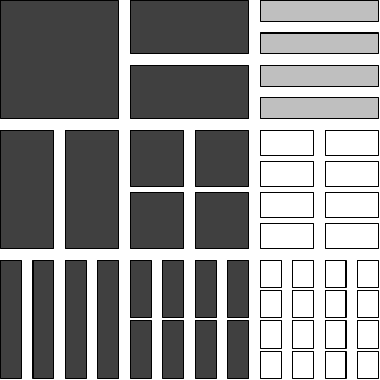}
        \end{center}
      \end{minipage}

      \bigskip\noindent
      Note that $[0,1)\times [0,|J_0|)\drlesseq [0,1)\times J_0$.  Define $\lambda([0,1)\times J_0)$
      to be
      \begin{equation}\label{eq:case-2a-lambda}
        \lambda([0,1)\times J_0) =
        \max\{\nu(Q) : Q\drless [0,1)\times [0,|J_0|)\}.
      \end{equation}
      Recall that for a dyadic interval $L_0$ we denote its left (=lower) half by $L_0^\ell$ and its
      right (=upper) half by $L_0^r$.  The set $Y_{[0,1)\times\widetilde{J}_0}$ has already been
      defined.  Now, put
      \begin{equation*}
        Y_{[0,1)\times J_0} =
        \begin{cases}
          \bigcup\{ L_0^\ell : L_0\in\mathscr{D}_{\lambda([0,1)\times J_0)},\, L_0\subset
          Y_{[0,1)\times \widetilde{J}_0}\}
          &\text{if}\ \ J_0 = \widetilde{J}_0^\ell,\\
          \bigcup\{ L_0^r : L_0\in\mathscr{D}_{\lambda([0,1)\times J_0)},\, L_0\subset
          Y_{[0,1)\times\widetilde{J}_0}\} &\text{if}\ \ J_0 = \widetilde{J}_0^r.
        \end{cases}
      \end{equation*}
      We define the family of high frequency covers of the set $[0,1)\times Y_{[0,1)\times J_0}$ by
      \begin{equation}\label{eq:rectangles-case_2}
        \mathscr F_m = \{[0,1)\times L\in \drec\, :\, L\in\dint_m,\, L\subset Y_{[0,1)\times J_0} \},
      \end{equation}
      for all $m > \lambda([0,1)\times J_0)$, see Figure~\ref{fig:building_blocks-2}, and observe
      that
      \begin{equation}\label{eq:rectangles-case_2a:pointset}
        \bigcup \mathscr{F}_m = [0,1)\times Y_{[0,1)\times J_0}.
      \end{equation}
      \begin{figure}[H]
        \begin{center}
          \includegraphics{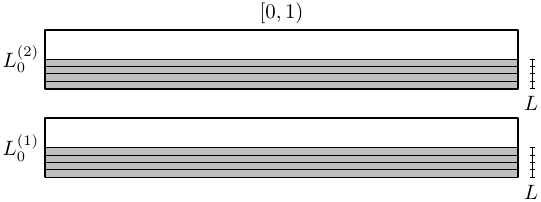}
        \end{center}
        \caption{The above figure depicts an instance of $\mathscr F_m$ in
          Case~2.\textsc{a}. $L_0^{(k)}$ is a dyadic interval such that
          $[0,1)\times L_0^{(k)}\in \mathscr B_{[0,1)\times \widetilde J_0}$, and $L$ is a dyadic
          interval such that $[0,1)\times L\in \mathscr F_m$.}
        \label{fig:building_blocks-2}
      \end{figure}

      \noindent
      \begin{minipage}[H]{.75\textwidth}
        \textsc{Case~2.b: $I_0 \neq [0,1)$.}  The collections indexed
        by the black rectangles have already been constructed.  Here, we determine the collections
        for the gray rectangles.
      \end{minipage}
      \begin{minipage}[H]{.25\textwidth}
        \begin{center}
          \includegraphics[scale=0.4]{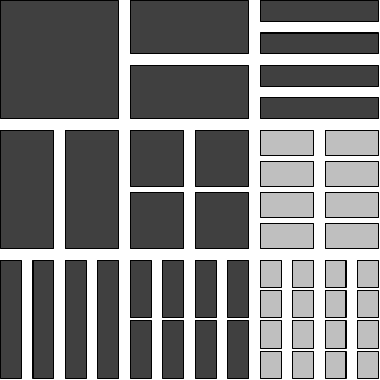}
        \end{center}
      \end{minipage}

      \bigskip\noindent By our induction hypothesis,
      $\{\mathscr{B}_{I\times J} : \drindex(I\times J) \leq i_0-1\}$
      satisfies~(R\ref{enu:r1})--(R\ref{enu:r6}).  At this stage of the proof, the set
      $Y_{[0,1)\times J_0}$ has already been constructed.  Now, we define the high frequency covers
      of $X_{I_0\times J_0}\times Y_{[0,1)\times J_0}$ by putting
      \begin{equation}\label{eq:rectangles-case_2b}
        \mathscr F_m
        = \{K_0\times L\, :\,
        K_0\in \mathscr X_{I_0\times J_0},\,
        L\in\dint_m,\, L\subset Y_{[0,1)\times J_0}
        \},
      \end{equation}
      whenever $m > \nu(\widetilde I_0\times J_0)$, see Figure~\ref{fig:building_blocks-3}.
      \begin{figure}[bht]
        \begin{center}
          \includegraphics{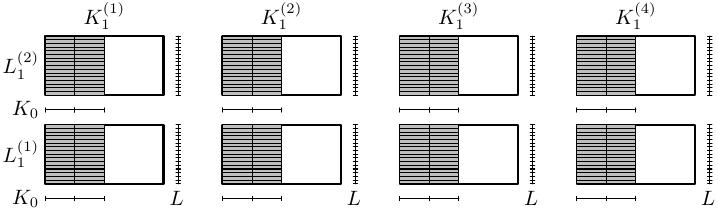}
        \end{center}
        \caption{The above figure depicts an instance of $\mathscr F_m$ in Case~2.\textsc{b}.  We
          have $K_1^{(k)}\in\mathscr X_{\widetilde I_0\times J_0}$,
          $L_1^{(\ell)}\in\mathscr Y_{\widetilde I_0\times J_0}$, and the dyadic interval $K_0$ is
          in $\mathscr X_{I_0\times J_0}$.  $\mathscr F_m$ is the collection of all the small gray
          rectangles.  We obtain $\mathscr{F}_m$ by leaving intact the intervals of the
          $x$-coordinate ($K_0\in \mathscr X_{I_0\times J_0}$) and using a high frequency cover --
          comprised of the intervals $L$ -- of the intervals
          $L_1^{(\ell)}\in \mathscr Y_{\widetilde I_0\times J_0}$.  The intervals
          $L_1^{(\ell)}\in \mathscr Y_{\widetilde I_0\times J_0}$ in this Figure are covering the
          exact same set as the intervals denoted by $L$ in Figure~\ref{fig:building_blocks-2},
          i.e. they cover $Y_{[0,1)\times J_0}$.}\label{fig:building_blocks-3}
      \end{figure}
    \end{proofcase}

    In each of the above
    cases~\eqref{eq:rectangles-case_1},~\eqref{eq:rectangles-case_1b},~\eqref{eq:rectangles-case_2},
    and~\eqref{eq:rectangles-case_2b} we define the following functions.  Firstly, let
    \begin{subequations}\label{eq:block-basis-candidate}
      \begin{equation}\label{eq:block-basis-candidate:a}
        f_m = \sum_{Q\in \mathscr F_m} h_Q,
      \end{equation}
      and secondly for any choice of signs $\varepsilon_Q \in \{-1,+1\}$, $Q\in \mathscr F_m$ put
      \begin{equation}\label{eq:block-basis-candidate:b}
        f_m^{(\varepsilon)} = \sum_{Q\in \mathscr F_m} \varepsilon_Q h_Q.
      \end{equation}
    \end{subequations}

    Now, we specify the value of $m$.  To this end, put
    \begin{equation}\label{eq:induction-step:disjoint_spectrum:0}
      k_{i_0} = \max\{ \mu(R),\nu(R) : R\in\drec,\ \drindex(R)\leq i_0-1\},
    \end{equation}
    and note that each $\mathscr F_m$, $m > k_{i_0}$ can be written as the product of two sets of
    intervals, i.e.
    \begin{equation*}
      \mathscr F_m = \{ K\times L : K\in\mathscr X_m,\ L\in \mathscr Y_m \},
      \qquad m > k_{i_0},
    \end{equation*}
    where the collections $\mathscr X_m$ and $\mathscr Y_m$, $m > k_{i_0}$, satisfy the following:
    \begin{itemize}
    \item $\mathscr{X}_m$ and $\mathscr{Y}_m$ are each a non-empty, finite collection of pairwise
      disjoint dyadic intervals of equal length, whenever $m > k_{i_0}$;
    \item $\mathscr{X}_m\cap\mathscr{X}_n=\emptyset$ or $\mathscr{Y}_m\cap\mathscr{Y}_n=\emptyset$
      whenever $m,n > k_{i_0}$ are distinct;
    \item the union of the sets in~$\mathscr{X}_m$ is independent of~$m > k_{i_0}$, and the
      union of the sets in~$\mathscr{Y}_m$ is independent of~$m > k_{i_0}$.
    \end{itemize}
    Thus, by Lemma~\ref{lem:thesequencefm}, we have that
    \begin{itemize}
    \item for each $g\in H^p(H^q)^*$, $\sup_{\gamma\in\Gamma}|\langle M_\gamma f_m, g\rangle|\to 0$
      as $m\to\infty;$
    \item for each $g\in H^p(H^q)$, $\sup_{\gamma\in\Gamma}|\langle M_\gamma g, f_m\rangle|\to 0$ as
      $m\to\infty$;
    \end{itemize}
    where we recall that $\Gamma$ denotes the unit ball of $\ell^\infty(\drec)$, and that
    $\gamma = (\gamma_R : R\in \drec)\in \Gamma$ defines the operator $M_\gamma$
    (see~\eqref{eq:multiplier}).  Hence, we can find an integer $m_{i_0} > k_{i_0}$ such that
    \begin{equation}\label{eq:induction-step:a}
      \sum_{j=0}^{i_0-1} |\langle T b_j^{(\varepsilon)}, f_{m_{i_0}}^{(\varepsilon)}\rangle|
      + |\langle f_{m_{i_0}}^{(\varepsilon)}, T^*b_j^{(\varepsilon)}\rangle|
      \leq \eta \delta 4^{-i_0-2},
    \end{equation}
    for all choices of signs $\varepsilon_{K\times L}$, $K\times L \in \mathscr F_{m_{i_0}}$.  Now,
    we put
    \begin{equation}\label{eq:induction-step:block_basis_collection:dfn}
      \mathscr B_{I_0\times J_0} = \mathscr B_{i_0} = \mathscr F_{m_{i_0}}.
    \end{equation}
    \begin{subequations}\label{eq:induction-step:block_basis_collection}
      If $I_0\times J_0$ is a ``Case~1'' rectangle, i.e. $|I_0| < |J_0|$, then
      \begin{equation}\label{eq:induction-step:block_basis_collection:a}
        \mu(I_0\times J_0) = m_{i_0}
        \quad\text{and}\quad
        \mathscr{X}_{I_0\times J_0}
        = \{ I\in\mathscr{D}_{m_{i_0}} : I\times J\in \mathscr{B}_{I_0\times J_0}\},
      \end{equation}
      and if $I_0\times J_0$ is a ``Case~2'' rectangle, i.e. $|I_0| \geq |J_0|$, then
      \begin{equation}\label{eq:induction-step:block_basis_collection:b}
        \nu(I_0\times J_0) = m_{i_0}
        \quad\text{and}\quad
        \mathscr{Y}_{I_0\times J_0}
        = \{ J\in\mathscr{D}_{m_{i_0}} : I\times J\in \mathscr{B}_{I_0\times J_0}\}.
      \end{equation}
    \end{subequations}
    Thereby, we have completed the construction of $\mathscr B_{I_0\times J_0} = \mathscr B_{i_0}$.

    Reviewing the four cases Case~1.\textsc{a}, Case~1.\textsc{b}, Case~2.\textsc{a}, and
    Case~2.\textsc{b} of the construction we see that $\{\mathscr{B}_i : i\leq i_0\}$
    satisfies~(R\ref{enu:r1})--(R\ref{enu:r6}).
  \end{proofstep}

  \begin{proofstep}[Selecting the signs $\varepsilon$]
    Let $\varepsilon_{Q}\in\{\pm 1\}$, $Q \in \mathscr{B}_{i_0}$ be fixed.  We obtain
    from~\eqref{eq:decomp} and~\eqref{eq:block-basis-candidate}
    \begin{equation*}
      \langle T f_{m_{i_0}}^{(\varepsilon)}, f_{m_{i_0}}^{(\varepsilon)} \rangle
      = \sum_{Q\in \mathscr{B}_{i_0}} \alpha_{Q} |Q|
      + \langle f_{m_{i_0}}^{(\varepsilon)}, s_{m_{i_0}}^{(\varepsilon)} \rangle,
    \end{equation*}
    where
    \begin{equation*}
      s_{m_{i_0}}^{(\varepsilon)} = \sum_{Q\in \mathscr{B}_{i_0}}
      \varepsilon_{Q} r_{Q}.
    \end{equation*}
    By~\eqref{eq:decomp} we have $\langle h_{Q}, r_{Q} \rangle = 0$, $Q\in\drec$, and consequently
    \begin{equation}\label{eq:off-diag-rand}
      \langle f_{m_{i_0}}^{(\varepsilon)}, s_{m_{i_0}}^{(\varepsilon)} \rangle
      = \sum \varepsilon_{Q_0} \varepsilon_{Q_1}
      \langle h_{Q_0}, r_{Q_1} \rangle,
    \end{equation}
    where the sum is taken over all $Q_0, Q_1\in \mathscr{B}_{i_0}$ with $Q_0\neq Q_1$.  Let
    $\cond_\varepsilon$ denote the average over all possible choices of signs $\varepsilon_{Q}$,
    $Q\in \mathscr{B}_{i_0}$.  Taking expectations we obtain from~\eqref{eq:off-diag-rand} that
    \begin{equation*}
      \cond_\varepsilon \langle f_{m_{i_0}}^{(\varepsilon)}, s_{m_{i_0}}^{(\varepsilon)} \rangle = 0.
    \end{equation*}
    This gives us
    \begin{equation*}
      \cond_\varepsilon \langle T f_{m_{i_0}}^{(\varepsilon)}, f_{m_{i_0}}^{(\varepsilon)} \rangle
      = \sum_{Q\in \mathscr{B}_{i_0}} \alpha_{Q} |Q|.
    \end{equation*}
    Hence, in view of~\eqref{eq:a-estimate}, there exists at least one $\varepsilon$ such that
    \begin{equation}\label{eq:diagonal_estimate-1}
      |\langle T f_{m_{i_0}}^{(\varepsilon)}, f_{m_{i_0}}^{(\varepsilon)} \rangle|
      \geq \sum_{Q\in \mathscr{B}_{i_0}} \alpha_{Q} |Q| \ge \delta \|f_{m_{i_0}}^{(\varepsilon)}\|_2^2.
    \end{equation}
    We complete the inductive construction by choosing $\varepsilon$ according
    to~\eqref{eq:diagonal_estimate-1} and define
    \begin{equation}\label{eq:induction-step:block_basis}
      b_{I_0\times J_0} ^{(\varepsilon)}
      = b_{i_0}^{(\varepsilon)}
      = f_{m_{i_0}}^{(\varepsilon)}.
    \end{equation}
    Hence,~\eqref{eq:induction-properties:c} holds for $i=i_0$, while~\eqref{eq:induction-step:a}
    ensures that~\eqref{eq:induction-properties:b} holds for $i=i_0$.
  \end{proofstep}

  \begin{proofstep}[Essential properties of our inductive construction]
    Since each of the finite collections $\{\mathscr{B}_i : i\leq i_0\}$, $i_0\in\mathbb{N}_0$,
    satisfies~(R\ref{enu:r1})--(R\ref{enu:r6}), Remark~\ref{rem:aux-condition} asserts that the
    infinite collection $\{\mathscr{B}_i : i\in\mathbb{N}_0\}$ satisfies
    (R\ref{enu:r1})--(R\ref{enu:r6}), and hence, by Lemma~\ref{conditionsensuringP1P4}, it satisfies
    the local product condition (P\ref{enu:p1})--(P\ref{enu:p4}) with constants $C_X = C_Y=1$.

    For $1 \leq u,v < \infty$ and $I\times J\in \drec$,
    Lemma~\ref{lem:thesequencefm}\eqref{lem:thesequencefm:0}--\eqref{lem:thesequencefm:1} together
    with~\eqref{eq:XY-inclusions} and~(P\ref{enu:p3}) gives us the following mixed-norm estimates
    for $b_{I\times J}^{(\varepsilon)}$:
    \begin{subequations}\label{eq:mixed-norm_estimates}
      \begin{align}
        \|b_{I\times J}^{(\varepsilon)}\|_{H^u(H^v)}
        & = |I|^{1/u} |J|^{1/v} = \|h_{I\times J}\|_{H^u(H^v)},
          \label{eq:mixed-norm_estimates:a}\\
        \|b_{I\times J}^{(\varepsilon)}\|_{H^u(H^v)^*}
        & = |I|^{1-1/u} |J|^{1-1/v} = \|h_{I\times J}\|_{H^u(H^v)^*}.
          \label{eq:mixed-norm_estimates:b}
      \end{align}
    \end{subequations}

    The estimates~\eqref{eq:induction-properties:b} and~\eqref{eq:induction-properties:c} show that
    the block basis $\{b_i^{(\varepsilon)}\}$ almost-diagonalizes $T$ in the following precise
    sense:
    \begin{align}
      |\langle T b_i^{(\varepsilon)}, b_i^{(\varepsilon)} \rangle|
      & \geq \delta \|b_i^{(\varepsilon)}\|_2^2,
        \qquad i\in \mathbb{N}_0,
        \label{eq:almost-eigenvectors:large_diagonal}\\
      \sum_{j=0}^{i-1} |\langle T b_j^{(\varepsilon)}, b_i^{(\varepsilon)}\rangle| + |\langle
      T b_i^{(\varepsilon)}, b_j^{(\varepsilon)}\rangle|
      & \leq\eta\delta 4^{-i-2},
        \qquad i\in\mathbb{N}.
        \label{eq:almost-eigenvectors:0}
    \end{align}
  \end{proofstep}

  \begin{proofstep}[Putting it together]
    The basic model of argument presented below can be traced to the seminal paper of Alspach,
    Enflo, and Odell~\cite{alspach_enflo_odell:1977}.  Since $\{\mathscr B_{I\times J}\}$ satisfies
    the local product condition~(P\ref{enu:p1})--(P\ref{enu:p4}) with constants $C_X=C_Y=1$, we
    obtain from Theorem~\ref{thm:projection} the following.  First, let
    $Y = \overline{\spn}\{b_i^{(\varepsilon)}\,:\,i\in\mathbb{N}_0\}\subset H^p(H^q)$ and let
    $B_\varepsilon : H^p(H^q)\to Y$ denote the unique linear extension of
    $B_\varepsilon h_i = b_i^{(\varepsilon)}$, $i\in\mathbb{N}_0$, then by
    Theorem~\ref{thm:projection}
    \begin{equation}\label{eq:commutative-diagram:preimage}
      \vcxymatrix{H^p(H^q) \ar[r]^{I_{H^p(H^q)}} \ar[d]_{B_\varepsilon} & H^p(H^q)\\
        Y \ar[r]_{I_Y} & Y \ar[u]_{A_\varepsilon|Y}}
      \qquad \|B_\varepsilon\|=\|A_\varepsilon\| = 1,
    \end{equation}
    where we recall that $A_\varepsilon : H^p(H^q)\to H^p(H^q)$ denotes the operator given by
    \begin{equation*}
      A_\varepsilon f
      = \sum_{i=0}^\infty \frac{\langle f, b_i^{(\varepsilon)}\rangle}{\|b_i\|_2^2} h_i,
      \qquad f\in H^p(H^q).
    \end{equation*}

    Secondly, we put
    \begin{equation*}
      \gamma_i
      = \frac{\|b_i\|_2^2}{\langle T b_i^{(\varepsilon)}, b_i^{(\varepsilon)}\rangle},
      \qquad i\in\mathbb{N}_0.
    \end{equation*}
    Recall that $M_\gamma$ was defined in~\eqref{eq:multiplier} as the linear extension of
    $M_\gamma h_i = \gamma_i h_i$, $i\in\mathbb{N}_0$.  The operator norm of $M_\gamma$ is
    $\sup_{i\in\mathbb{N}_0} |\gamma_i| \leq \frac{1}{\delta}$
    by~\eqref{eq:almost-eigenvectors:large_diagonal}.  Define $U : H^p(H^q)\to Y$ by
    $U = B_\varepsilon M_\gamma A_\varepsilon$ and note that
    \begin{equation}\label{eq:almost-inverse}
      U(f) = \sum_{i=0}^\infty
      \frac{\langle f, b_i^{(\varepsilon)}\rangle}
      {\langle Tb_i^{(\varepsilon)}, b_i^{(\varepsilon)}\rangle}
      b_i^{(\varepsilon)},
      \qquad f\in H^p(H^q).
    \end{equation}
    The above estimates for the norms of the operators $A_\varepsilon$, $B_\varepsilon$, and $M_\gamma$ yield
    \begin{equation}\label{eq:almost-inverse:bounded}
      \|U : H^p(H^q)\to Y\|_{H^p(H^q)}
      \leq \|M_\gamma\|\, \|B_\varepsilon\|\, \|A_\varepsilon\|
      \leq \frac{1}{\delta}.
    \end{equation}

    Thirdly, observe that for all $g = \sum_{i=0}^\infty \lambda_i b_i^{(\varepsilon)} \in Y$, we
    have the identity
    \begin{equation}\label{eq:almost-inverse:identity}
      UTg - g
      = \sum_{i=1}^\infty \sum_{j=0}^{i-1} \lambda_j
      \frac{\langle T b_j^{(\varepsilon)}, b_i^{(\varepsilon)}\rangle}
      {\langle Tb_i^{(\varepsilon)}, b_i^{(\varepsilon)}\rangle}
      b_i^{(\varepsilon)}
      + \lambda_i
      \frac{\langle T b_i^{(\varepsilon)}, b_j^{(\varepsilon)}\rangle}
      {\langle Tb_j^{(\varepsilon)}, b_j^{(\varepsilon)}\rangle}
      b_j^{(\varepsilon)}.
    \end{equation}
    Using that $\|b_j^{(\varepsilon)}\|_{H^p(H^q)}\leq 1$, $j\in\mathbb{N}_0$, we obtain
    \begin{equation}\label{eq:U-estimate:1}
      \|UTg - g\|_{H^p(H^q)}
      \leq  \sum_{i=1}^\infty \sum_{j=0}^{i-1} |\lambda_j|
      \frac{|\langle T b_j^{(\varepsilon)}, b_i^{(\varepsilon)}\rangle|}
      {|\langle Tb_i^{(\varepsilon)}, b_i^{(\varepsilon)}\rangle|}
      + |\lambda_i|
      \frac{|\langle T b_i^{(\varepsilon)}, b_j^{(\varepsilon)}\rangle|}
      {|\langle Tb_j^{(\varepsilon)}, b_j^{(\varepsilon)}\rangle|}.
    \end{equation}
    Now, we will make the following two observations: The first is
    that~\eqref{eq:mixed-norm_estimates:b} implies $\|b_j^{(\varepsilon)}\|_{H^p(H^q)^*} \leq 1$,
    and thus by~\eqref{eq:mixed-norm_estimates:a} and~\eqref{eq:ordering-estimate}, we obtain
    \begin{equation*}
      \|g\|_{H^p(H^q)}
      \geq \|g\|_{H^p(H^q)} \|b_j^{(\varepsilon)}\|_{H^p(H^q)^*}
      \geq |\langle g, b_j^{(\varepsilon)} \rangle|
      = |\lambda_j| \|b_j^{(\varepsilon)}\|_2^2
      \geq \frac{1}{(1+\sqrt{j})^2} |\lambda_j|,
    \end{equation*}
    for all $j\in\mathbb{N}_0$.  The second observation is that
    $|\langle T b_j^{(\varepsilon)}, b_j^{(\varepsilon)} \rangle|\geq
    \frac{\delta}{(1+\sqrt{j})^2}$, $j\in\mathbb{N}_0$, which is a consequence
    of~\eqref{eq:almost-eigenvectors:large_diagonal},~\eqref{eq:mixed-norm_estimates:a},
    and~\eqref{eq:ordering-estimate}.  These two observations yield the following estimate:
    \begin{equation*}
      \frac{|\lambda_j|}{|\langle T b_i^{(\varepsilon)}, b_i^{(\varepsilon)}\rangle|}
      \leq \frac{1}{\delta}\|g\|_{H^p(H^q)} (1+\sqrt{i})^2(1+\sqrt{j})^2,
      \qquad j \neq i.
    \end{equation*}
    Inserting this estimate into~\eqref{eq:U-estimate:1} and
    applying~\eqref{eq:almost-eigenvectors:0} yields
    \begin{align*}
      \|UTg - g\|_{H^p(H^q)}
      & \leq \frac{1}{\delta} \|g\|_{H^p(H^q)} \sum_{i=1}^\infty (1+\sqrt{i})^4\sum_{j=0}^{i-1}
        |\langle T b_j^{(\varepsilon)}, b_i^{(\varepsilon)}\rangle|
        + |\langle T b_i^{(\varepsilon)}, b_j^{(\varepsilon)}\rangle|\\
      & \leq \frac{\eta}{2} \|g\|_{H^p(H^q)}.
    \end{align*}
    To see the latter estimate note that
    $\sum_{i=1}^\infty (1+\sqrt{i})^4 4^{-i} \leq 4 \sum_{i=1}^\infty (1+i)^2 4^{-i} =
    4\frac{53}{27}\leq 8$.

    Finally, let $J : Y\to H^p(H^q)$ denote the inclusion operator given by $Jy = y$. Since we
    assumed that $0 < \eta\le 1$, the operator $UTJ$ is invertible, and
    its inverse has norm at most $(1-\frac{\eta}{2})^{-1}\le 1+\eta$.  Now we define the operator
    $V : H^p(H^q)\to Y$ by $(UTJ)^{-1}U$ and observe that
    \begin{equation}\label{eq:commutative-diagram:image}
      \vcxymatrix{%
        Y \ar[rr]^{I_Y} \ar[dd]_J \ar[rd]_{UTJ} & & Y\\
        & Y \ar[ru]^{(UTJ)^{-1}} &\\
        H^p(H^q) \ar[rr]_T & & H^p(H^q) \ar[lu]_U \ar[uu]_V
      }
      \qquad \|J\|\|V\| \leq (1+\eta)/\delta.
    \end{equation}
    Merging the commutative diagram~\eqref{eq:commutative-diagram:preimage}
    with~\eqref{eq:commutative-diagram:image} yields
    \begin{equation*}
      \vcxymatrix{%
        H^p(H^q) \ar[rr]^{I_{H^p(H^q)}} \ar[d]_{B_\varepsilon} & & H^p(H^q)\\
        Y \ar[rr]^{I_Y} \ar[dd]_J \ar[rd]_{UTJ} & & \ar[u]_{A_\varepsilon|Y}Y\\
        & Y \ar[ru]^{(UTJ)^{-1}} &\\
        H^p(H^q) \ar[rr]_T & & H^p(H^q) \ar[lu]_U \ar[uu]_V
      }
    \end{equation*}
    where $\|B_\varepsilon\|=\|A_\varepsilon\| = 1$ and $\|J\|\|V\| \leq (1+\eta)/\delta$, which
    concludes the proof of Theorem~\ref{thm:2d-andrew}.\qedhere
  \end{proofstep}
\end{proof}

\bibliographystyle{abbrv}
\bibliography{bibliography}

\end{document}